%% file: PDGbabyverma.tex
\documentclass[12pt]{amsart}
\usepackage{
  a4wide, 
  amsmath,
  amsfonts,
  amssymb,
  graphicx, 
  times,
  color,
  hyphenat,
  enumitem,     
  mathabx, 
  stmaryrd, datetime, bbm, tikz-cd, thmtools, verbatim
}
\usepackage[cal=boondoxo]{mathalfa}

\usepackage{tikz,bbm,xspace}
\usepackage{hyperref}
\usepackage{cleveref}
\hypersetup{colorlinks=true, pdfstartview=FitV, linkcolor=blue, citecolor=blue, urlcolor=blue}

\definecolor{bordeaux}{rgb}{0.6,0,0.3}
\definecolor{rose}{rgb}{0.8,0,0.6}
\definecolor{darkmagenta}{rgb}{0.5,0,0.6}
\definecolor{dmarine}{rgb}{0.3,0.5,0.8}
\definecolor{gmarine}{rgb}{0.1,0.6,0.3}
\definecolor{dblue}{rgb}{0,0,0.5}
\definecolor{dred}{rgb}{0.9,0,0}
\definecolor{dgreen}{rgb}{0,0.6,0}
\definecolor{myred}{rgb}{0.9,0,0}
\definecolor{mygreen}{rgb}{0,0.7,0}
\definecolor{myblue}{rgb}{0,.5,1}
\newcommand{\nnfootnote}[1]{%
\begin{NoHyper}
\renewcommand\thefootnote{}\footnote{#1}%
\addtocounter{footnote}{-1}%
\end{NoHyper}
}
\input xy
\xyoption{all}
\input defs.tex

\newcommand{\fdot}[3][]{ \node  [anchor = center, fill=white, draw=black,circle,inner sep=2pt] at (#3) {} ; \node[xshift=-.065cm, yshift=-.065cm,anchor = south west] at (#3){\small $#1$}; \node[xshift=-.065cm, yshift=.065cm,anchor = north west] at (#3){\small $#2$}; }
\allowdisplaybreaks
\title{Enhanced nilHecke algebras and baby Verma modules
}
\author{Elia Rizzo}
\address{E.R.: Institut de Recherche en Math\'ematique et Physique\\
Universit\'e Catholique de Louvain\\ 
Chemin du Cyclotron 2\\ 
1348 Louvain-la-Neuve\\ 
Belgium}
\email{eliarizzo@hotmail.it}
\author{Pedro Vaz}
\address{P.V.: Institut de Recherche en Math{\'e}matique et Physique, 
Universit{\'e} Catholique de Louvain, Chemin du Cyclotron 2,  
1348 Louvain-la-Neuve, Belgium \\
\href{https://orcid.org/0000-0001-9422-4707}{ORCID 0000-0001-9422-4707},
\href{https://perso.uclouvain.be/pedro.vaz}{https://perso.uclouvain.be/pedro.vaz}} 
\email{pedro.vaz@uclouvain.be}

\begin{document}
\usetikzlibrary{matrix, calc, arrows}
 \usetikzlibrary{decorations.pathreplacing,backgrounds,decorations.markings}

\tikzstyle directed=[postaction={decorate,decoration={markings,
    mark=at position #1 with {\arrow{>}}}}]
\tikzstyle rdirected=[postaction={decorate,decoration={markings,
    mark=at position #1 with {\arrow{<}}}}]

 \tikzset{wei/.style={draw=red,double=red!40!white,double distance=1.5pt,thin}}
\tikzset{bdot/.style={fill,circle,color=blue,inner sep=3pt,outer sep=0}}
\tikzset{anchorbase/.style={baseline={([yshift=-0.5ex]current bounding box.center)}},
tinynodes/.style={font=\tiny,text height=0.25ex,text depth=0.05ex},
smallnodes/.style={font=\scriptsize,text height=0.75ex,text depth=0.15ex},
usual/.style={line width=0.9,color=black},
dusual/.style={line width=0.9,color=tomato,densely dashed}
}
%
%
%
\begin{abstract}
We define a derivation on the enhanced nilHecke algebra yielding a $p$-dg algebra when working over a field of characteristic $p$.  
We define functors on the category of $p$-dg modules resulting in an action of small quantum $\slt$ on the Grothendieck group, which is isomorphic to a baby Verma module. 
We upgrade the derivation into an action of the Lie algebra $\cslt$.
\end{abstract}
\nnfootnote{\textit{Mathematics Subject Classification 2020.} Primary: 18N25, 16W25; Secondary: 20G42, 20C08.}
\nnfootnote{\textit{Keywords.} Categorified Verma modules, $p$-DG algebras and modules, $\cslt$-actions, baby Verma modules.}
\maketitle

\tableofcontents
\pagestyle{myheadings}
\markboth{\em\small Elia Rizzo and Pedro Vaz}{\em\small $p$-dg enhanced nilHecke and Baby Vermas}
%
%
%
\section*{Introduction}

\input{files/intro.tex}
\section{Reminders}\label{sec:reminders}
\input{files/reminders.tex}
\section{The \texorpdfstring{$p$}{p}-dg enhanced nilHecke algebra}\label{sec:pdgAn}
\input{files/pdgAn.tex}
\section{The Grothendieck group of \texorpdfstring{$(A,\dif)$}{(A,d)} and baby Verma modules}
\input{files/cataction.tex}
\section{An action of \texorpdfstring{$\cslt$}{sl2}}
\input{files/sl2.tex}



\vspace*{1cm}

\input{files/biblio.tex}


\end{document}

%% file: defs.tex
\newcounter{saveenumerate}
\makeatletter
\newcommand{\enumeratext}[1]{%
\setcounter{saveenumerate}{\value{enum\romannumeral\the\@enumdepth}}
\end{enumerate}
#1
\begin{enumerate}
\setcounter{enum\romannumeral\the\@enumdepth}{\value{saveenumerate}}%
}
\makeatother

\newcommand{\rnz}{R_n}
\newcommand{\dif}{\mathrm{d}}
\newcommand{\hs}{\mathcal{h}}
\newcommand{\es}{\mathcal{e}}
\newcommand{\cglt}{\mathcal{gl}_2}
\newcommand{\cslt}{\mathcal{sl}_2}
\newcommand{\csln}{\mathcal{sl}_n}
\newcommand{\smslt}{u_q(\slt)}
\newcommand{\kp}{\Bbbk_p}
\newcommand{\ee}{\mathtt{e}}
\newcommand{\ff}{\mathtt{f}}
\newcommand{\hh}{\mathtt{h}}

\newcommand{\Sy}{\mathfrak{S}}

\DeclareMathOperator{\nh}{NH}
\DeclareMathOperator{\nc}{NC}

\DeclareMathOperator{\mat}{Mat}

\newcommand{\bV}{\raisebox{0.03cm}{\mbox{\footnotesize$\textstyle{\bigwedge}$}}}

\newcommand{\slt}{{\mathfrak{sl}_{2}}}

\newcommand{\und}[1]{{\underline{#1}}}

\newcommand{\cP}{\mathcal{P}}

\DeclareMathOperator{\amod}{\mathrm{-}mod}

\DeclareMathOperator{\stmod}{\mathrm{-}\underline{mod}}

\DeclareMathOperator{\qdeg}{qdeg}

\DeclareMathOperator{\End}{End}

\newcommand{\tikzdiagh}[2][]{\tikz[#1,very thick,baseline={([yshift=#2]current bounding box.center)}]}
\newcommand{\tikzdiag}[1][]{\tikzdiagh[#1]{-.5ex}}
\tikzstyle{tikzdot}=[fill=black, circle, inner sep=2pt]
\newcommand{\plusspacing}{\llap{\phantom{\small ${+}1$}}}

\theoremstyle{plain}
\newtheorem{thm}{Theorem}[section]
\newtheorem{cor}[thm]{Corollary}
\newtheorem{lem}[thm]{Lemma}
\newtheorem{prop}[thm]{Proposition}
\newtheorem{conj}[thm]{Conjecture}

\theoremstyle{definition}
\newtheorem{rem}[thm]{Remark}

\newtheorem{ex}[thm]{Example}

\theoremstyle{definition}
\newtheorem{defn}[thm]{Definition}

\newcommand{\bN}{\mathbb{N}}
\newcommand{\bZ}{\mathbb{Z}}

\newcommand{\bC}{\mathbb{C}}
\newcommand{\bF}{\mathbb{F}}

\newcommand{\bO}{\mathbb{O}}

\newcommand{\cC}{\mathcal{C}}
\newcommand{\cD}{\mathcal{D}}
\newcommand{\cE}{\mathcal{E}}
\newcommand{\cF}{\mathcal{F}}

\newcommand{\cK}{\mathcal{K}}

\newcommand{\cR}{\mathcal{R}}

\DeclareMathOperator{\id}{Id}

\DeclareMathOperator{\Pol}{Pol}
\DeclareMathOperator{\Sym}{Sym}

\newcommand{\proj}{\twoheadrightarrow}

\catcode`\@=11
\long\def\@makecaption#1#2{%
    \vskip 10pt
    \setbox\@tempboxa\hbox{%
\small{#1: }\ignorespaces #2}%
    \ifdim \wd\@tempboxa >\captionwidth {%
        \rightskip=\@captionmargin\leftskip=\@captionmargin
        \unhbox\@tempboxa\par}%
      \else
        \hbox to\hsize{\hfil\box\@tempboxa\hfil}%
    \fi}
\newdimen\@captionmargin\@captionmargin=2\parindent
\newdimen\captionwidth\captionwidth=\hsize
\catcode`\@=12
\def\makeautorefname#1#2{\expandafter\def\csname#1autorefname\endcsname{#2}}
\makeautorefname{lem}{Lemma}%
\makeautorefname{prop}{Proposition}%
\makeautorefname{rem}{Remark}%

%% file: files/intro.tex

When $q$ is a root of unity, new phenomena appear in the representation theory of quantum Kac--Moody algebras. 
There are several versions of the quantum group at a root of unity and we are interested in the finite-dimensional version $\smslt$ (see~\autoref{def:smallqg} below) called ``small quantum group'' in several places of the literature (see for example~\cite{KhQi} and~\cite{AndersenMazorchuk}) and something different in others (for example, it is called a ``semi-restricted quantum group'' in~\cite{murakami-nagatomo}).
One representation that is of interest to us is the so-called ``baby Verma module'', which the induced module defined as in~\cite[\S3.2]{AndersenMazorchuk}:
\[
M(\lambda) := \smslt\otimes_{b_q}\bC_\lambda .
\]
Here as usual $b_q\subset\smslt$ is the (standard) Borel subalgebra, $\lambda$ is a generic parameter and $\bC_\lambda$ is a 1-dimensional representation of the Cartan subalgebra, extended trivially to $b_q$. The dimension of $M(\lambda)$ equals the order of $q$ and corresponds to the representation denoted $V(\lambda,a,b)$ in~\cite[Ex. 11.1.7]{chari-pressley} specialized to $\smslt$ i.e. with $a=b=0$. 

Two-dimensional baby Verma modules for $\smslt$ can be used to reproduce the Alexander link polynomial~\cite{murakami-alexanderpoly}. More generally, they have been successful in producing  new link invariants~\cite{deguchi} generalizing the Alexander polynomial and new interesting 3-manifold invariants~\cite{bgcpm-nonsemisimple}.

\subsubsection*{Categorification}
A major breakthrough in categorification was the development of Kac--Moody 2-categories, aka categorified quantum groups, by Lauda~\cite{Lauda1}, Khovanov--Lauda~\cite{KL1,KL2,KL3} and Rouquier~\cite{Rouquier-twoKM,rouquierquiv} and has undergone an impressive development (see for example~\cite{BK-isos,BK-graded,webster-knot,mackaay-webster,msv-schur,LQR,QR-foams}). 
The structures of Khovanov--Lauda and Rouquier's categorifications are tailor made for categorifying integrable representations of quantum Kac--Moody algebras and cannot be used to categorify non-integrable representations like for example Verma modules. 
The program on categorification of Verma modules was initiated by the second author together with Naisse in~\cite{naissevaz1} with the case of quantum $\slt$ and continued in the sequels~\cite{naissevaz2,naissevaz3,naissevaz4} (see~\cite{lanava-blob} and~\cite{DuNa-categorical-lkb} for further developments). 

The program of categorification at roots of unity that started with Khovanov in~\cite{Kh-hopfological} has been successfully applied  to categorification of quantum groups and its representations~\cite{KhQi,QiS-Burau,EliasQi1,EliasQi2,KhQiSussan,KhQiSussan2}, as well as to Hecke algebras and to link invariants~\cite{eliasqi-Hecke,QiSussan-jones,QiSussan-jones2,qrsw}. Part of the success of this program is explained to the ability of finding interesting $p$-nilpotent derivations on the structures used in categorification, as suggested in~\cite[\S 4]{Kh-hopfological}. In most examples of interest this takes us to work with $p$-dg algebras and $p$-dg categories, which are natural generalizations of dg-algebras and dg-categories that were studied in~\cite{Qi}. In~\cite{KhQi} a $p$-dg structure was constructed on the nilHecke algebra resulting in categorification of one half of $\smslt$ at a prime root of unity. This $p$-derivation was later upgraded into an action of the Lie algebra $\cslt$ by finding a ``counterdifferential'' and forming an $\cslt$-triple (see~\cite{eliasqi-sl2}).

\subsubsection*{What this paper does}
This work takes place in the confluence of categorification of Verma modules and Hopfological algebra. In the simplest case of $\slt$ the nilHecke algebra is replaced by an algebra $A_n$ which, as an abelian group is the tensor product of the nilHecke algebra $\nh_n$ and an exterior algebra in $n$ variables $\bV(\omega_1,\cdots,\omega_n)$. The defining relations of $A_n$ guarantee that $\nh_n$ and $\bV(\omega_1,\cdots,\omega_n)$ are subalgebras (see~\cite{naissevaz2} for details). 

The purpose of this paper is to extend to $A_n$ the $p$-derivation on $\nh_n$ and use it to categorify baby Verma modules. We remark that the defining relations of $A_n$ imply that a trivial extension of Khovanov--Qi's $p$-differential to $A_n$ cannot exist.
The requirement about the Grothendieck group is a baby Verma module imposes conditions on the $p$-derivations that we allow: for example, for each $n$ a good derivation has to act trivially on $\omega_n$ (since it cannot act trivially on all $\omega_i$s this indicates that it depends on $n$).
We construct a $p$-derivation $\dif$ on $A_n$, depending on $n$ and satisfying all the required conditions, it:
\begin{enumerate}
\item restricts to Khovanov-Qi's derivation on $\nh_n$,
\item is $\Sy_n$-equivariant,
\item \label{condd} acts trivially on $\omega_n$,
\item is compatible with an inclusion $A_n\hookrightarrow A_{n+1}$ extending the canonical inclusion $\nh_n\hookrightarrow\nh_{n+1}$, and
\item becomes $p$-nilpotent after base-change to a field of characteristic $p$.
\end{enumerate}
Condition \ref{condd} imposes severe conditions on a map including $A_n$ into $A_{n+1}$. See~\autoref{prop:Andiff} and~\autoref{prop:inclusion} for the explicit forms of the derivations and the inclusion maps.

\smallskip

Here is a brief plan of the paper. In~\autoref{sec:reminders} we give a short review on $p$-dg algebras and their $p$-dg modules, and of the $p$-dg nilHecke algebra and the enhanced nilHecke algebra $A_n$ that will be main players in the rest of the paper.   
In~\autoref{sec:pdgAn} we extend Khovanov and Qi's derivation to the enhanced nilHecke algebra $A_n$ and we construct a non canonical inclusion $A_n\hookrightarrow A_{n+1}$, compatible with our derivation. Contrary to Khovanov and Qi's, our derivation depends on $n$. During the process we extend nontrivially the notion of partial derivatives to $\Bbbk[x_1,\dotsc,x_n]\otimes\bV(\omega_1,\cdots,\omega_n)$.
In~\autoref{sec:cataction} we work over a field of characteristic $p$, where $A_n$ becomes a $p$-dg algebra. The inclusion $(A_n,\dif)\hookrightarrow (A_{n+1},\dif)$ allows defining functors $\cE$, $\cF$ and $\cK^{\pm 1}$ on a (derived) category of $p$-dg modules over $A:=\oplus_{n\geq 0}A_n$ that induce operators on the Grothendieck group $\mathbf{K}_0(A)$ satisfying the relations of quantum $\slt$. We further show that $\mathbf{K}_0(A)$ is isomorphic to a baby Verma module. Finally, in~\autoref{sec:sltwo} we use the partial derivatives to define an action of the positive Witt algebra on $R_n$. This allows upgrading the $p$-derivation to an action of $\cslt$ on $A_n$. 
We find that opposite to $\nh_n$, the algebra $A_n$ is not a weight module for $\cslt$, but it is filtered by weight modules instead.

\subsubsection*{Loose ends and speculations}
One important feature of the extended nilHecke algebra is the existence of a family of differentials $d_N$, indexed by the naturals. This results in a dg-algebra that is quasi-isomorphic to a cyclotomic nilHecke algebra, seen as a dg-algebra with zero differential (see~\cite[\S 8.6]{naissevaz1}). This is a fundamental ingredient to categorify the projection of the universal Verma module to the $N+1$-dimensional irreducible representation under the specialization $\lambda=q^N$.
The differential $d_N$ does not commute with the $p$-differentials, preventing us from relating our construction to the construction using cyclotomic quotients of the nilHecke algebra~\cite[\S 6.3]{EliasQi1}. 
A connection between these two constructions requires new ideas, lacking at the moment. 

In this paper we prove that the $\slt$-relations are satisfied on the Grothendieck group. It would be desirable to have a $p$-dg version of a categorical commutation relation, in the style of Corollary 3.16 in~\cite{naissevaz2}, yielding a categorical $\slt$-action. Ideally we would like to find a solution to this question, but so far it has eluded us. 

There are obvious follow-up questions about generalizing the results in this paper to constructions in link homology, like the ones in
\cite{naissevaz3,naissevaz4,lanava-blob}. In a different direction, being equipped with partial derivatives might help further understanding the structure of Verma categorification.

\subsubsection*{Acknowledgments} 
P.V. was supported by the Fonds de la Recherche Scientifique - FNRS under Grants no. MIS-F.4536.19 and CDR-J.0189.23.


%% file: files/reminders.tex
Fix once and for all the fields $\Bbbk$ and $\kp$ of characteristic zero and $p>0$ respectively.
In this paper we work over $\Bbbk$ almost always and specialize $\kp$ whenever strictly necessary.
Unadorned tensor products $\otimes$ denote tensor products over $\Bbbk$ or $\kp$, depending on the context.
Except when requiring $p$-nilpotency of derivations, and in~\autoref{sec:cataction}, all our results are valid over $\bZ$, and the reader can replace $\Bbbk$ by $\bZ$ safely. We use the usual notation $\bF_p$ for the finite field $\bZ / p\bZ$ of characteristic $p>0$.

\subsection{\texorpdfstring{$p$}{p}-dg algebras and modules}
In this section we review the basic definitions and notations for $p$-dg algebras and their categories of modules that we need in this paper (see~\cite[\S 2]{KhQi} and~\cite{Qi,Kh-hopfological} for further details).
Define $\bO_p \coloneq \bZ[q]/(1+q^2+\cdots +q^{2(p-1)})$.
\begin{defn}\label{def:pDGA}
Let $A\cong \oplus_{i\in \bZ}A^i$ be a $\bZ$-graded $\Bbbk$-algebra.
The degree $2$ $\Bbbk$-linear map $\dif : A\to A$ 
is a \emph{$p$-differential} if it is $p$-nilpotent, i.e. $\dif^p=0$,   
and it satisfies the Leibniz rule
\[
\dif(ab)=\dif(a)b + a\dif(b),
\]
for any $a$, $b\in A$. The datum $(A, \dif)$ is called a \emph{$p$-dg algebra}.
\end{defn}
The degree of the operator is chosen to be $2$ for categorification reasons.
\begin{defn}
	Let $(A,\dif )$ be a $p$-dg algebra, $M$ a graded left $A$-module and $\dif_M:M\to M$ a $\Bbbk$-linear endomorphism of vector spaces of degree $2$ that is $p$-nilpotent
\[
\dif_M^p=0,
\]
and for any $a\in A$, $m\in M$
\[
  \dif_M(am)=\dif(a)m+a\dif_M(m).
\] 
The datum $(M,\dif_M)$ is called a left \emph{$p$-dg module} over the $p$-dg algebra $A$.
\end{defn}
We will often omit the differential when speaking about $p$-dg algebras and modules when there is no risk of confusion.
For example we write just $A$ for the $p$-dg algebra $(A,\dif)$.
\begin{defn}
Let $(A,\dif )$ be a $p$-dg algebra and $(M,\dif_M)$, $(N,\dif_N)$ two $p$-dg $A$-modules. A \emph{morphism of $p$-dg modules} $f:M\to N$ is a morphism of $A$-modules which commutes with the differentials of $M$ and $N$:
\[
f(\dif_M(m))=\dif_N(f(m)),
\] 
for any $m\in M$.
\end{defn}
The collection of $p$-dg modules over a $p$-dg algebra and of grading preserving $p$-dg morphisms forms an abelian category we denote $A_\dif\amod$.

Let us now define the $\bZ$-graded Hopf algebra $H=\kp[\dif]/(\dif ^p)$, such that $\deg(\dif)=2$ and with comultiplication 
\[
  \Delta(\dif)=1\otimes \dif + \dif \otimes 1,
\]  
counit
\[
  \epsilon(\dif^i)=\delta_{0,i},
\]  
and antipode
\[
  S(\dif)=-\dif.
\]  
A $p$-dg algebra over $\kp$ is also a graded $H$-module algebra.

\begin{defn} \label{def:smashproduct}
    Let $B$ be a bialgebra and let $A$ be a left $B$-module algebra. The \emph{smash product algebra of $A$ and $B$}, denoted $A\# B$, is the vector space $A\otimes B$ with product
\[
 \left(a\otimes h\right)\left(b\otimes g\right)=\sum a(h_{(1)}\cdot b)\otimes h_{(2)}g, \quad a,b\in A,\ h,g\in B,
\]  
where the symbol $\cdot$ represents the action of $B$ on $A$, and the summation uses the Sweedler notation.
\end{defn}
When the bialgebra $B=H=\kp[\dif]/(\dif ^p)$ and $A$ is a $p$-dg algebra over $\kp$ we denote by $A_d=A\#H$ the smash product between $A$ and $H$.
The grading chosen for $A$ and $H$ implies that $A_d$ is graded. This choice of notation is justified by the fact that a $p$-dg module over $A$ is the same as a module over the smash product algebra $A\#H$. Indeed, following~\autoref{def:smashproduct} we observe that
$A_\dif= A\#\kp[\dif]/(\dif ^p)$ is the $\kp$ vector space $A\otimes \kp[\dif]/(\dif ^p)$ with multiplication given by the rules 
\begin{align*}
	(a\otimes1)(b\otimes 1) &= ab\otimes 1,\\
	(1\otimes h)(1\otimes g) &= 1\otimes hg,\\
	(a\otimes 1)(1 \otimes g) &= a \otimes g,\\
	(1\otimes \dif)(b\otimes 1) &= (b\otimes \dif)+\dif(b)\otimes 1,
\end{align*}
for $a,b\in A$, $h,g,  \dif  \in H=\kp[\dif]/(\dif ^p)$. Observe that the definition of $A_\dif$ depends on the $p$-dg algebra structure of $A$.
\begin{defn}
	Let $(A,\dif )$ be a $p$-dg algebra, $(M,\dif_M)$ and $(N,\dif_N)$ two $p$-dg $A$-modules. Two $p$-dg morphisms $f,g:M\to N$ are said to be \emph{homotopic} if there exists a morphism of $A$-modules $h:M\to N$ of degree $2-2p$ such that 
        \[
          f-g=\sum_{i=0}^{p-1} \dif_N^{p-1-i} \circ h \circ \dif^{i}_M.
	\]
	In particular, a morphism which is homotopic to the zero map is called \emph{null-homotopic}.
\end{defn}

Null-homotopic morphisms form an ideal in $A_\dif\amod$. Modding out by this ideal we obtain the \emph{homotopy} category $A_\dif\stmod$, often indicated by $\cC(A,\dif)$, which is triangulated.
Now consider the $p$-dg algebra $(\Bbbk,0)$.
Since every $p$-dg $(A,\dif)$-module is a $p$-dg $(\Bbbk,0)$-module there are two forgetful functors
\begin{align*}
	A_\dif\amod \rightarrow&\, \Bbbk_0 \amod, \\
	A_\dif\stmod \rightarrow&\, \Bbbk_0 \stmod.
\end{align*}
The objects of the categories $\Bbbk_0 \amod$ and $\Bbbk_0 \stmod$ are called \emph{p-complexes} in the literature. 
In particular, the isomorphism class in $\Bbbk_0 \stmod$ of a $p$-complex $P$ is referred to as the \emph{cohomology} of $P$. Similarly, for a $p$-dg $A$-module $M$ we can talk about its underlying $p$-complex when viewed as an object of $\Bbbk_0 \amod$, and consequently about its cohomology. This leads to the following definition:
\begin{defn}
Let $M$, $N$ be $p$-dg $A$-modules. A morphism $f:M\rightarrow N$ is said to be a \emph{quasi-isomorphism} if it is an isomorphism in the homotopy category $\Bbbk_{0}\stmod$. In particular, if $M \rightarrow 0$ or $0\rightarrow M$ are quasi-isomorphisms we say that $M$ is \emph{acyclic}.
\end{defn}
The quasi-isomorphisms in $A_\dif\stmod$ form a localizing class, inverting them leads to the \emph{derived category} $\cD(A,\dif)$, which we often denote $\cD(A)$ when no confusion can arise. This category is triangulated and Karoubian.

\subsection{\texorpdfstring{$p$}{p}-dg nilHecke algebras}
We continue with the review by covering the main definitions and results of~\cite[\S3.1-3.2 ]{KhQi}.
Let $\Pol_n\coloneq\Bbbk \left[x_1,\dotsc ,x_n\right]$ be the graded polynomial ring with $\deg(x_i)=2$. We write $\Sy_n$ for the symmetric group generated by the simple transpositions $s_1, \dotsc , s_{n-1}$, and $\Sym_n\coloneq\Pol_n^{\Sy_n}\subset\Pol_n$ for the (graded) subring of $\Sy_n$-invariants.
The balanced free left graded $\Pol_n$-module $\cP_n$ is
\[
\cP_n\coloneq\Pol_n\cdot v,
\]
where $v$ is the generator of $\cP_n$ with degree $\deg(v)=\frac{n(1-n)}{2}$.
The nilHecke algebra is the (graded) algebra of endomorphisms of the left $\Sym_n$-module $\cP_n$:
\begin{align}
\label{nh=end}
\nh_n\coloneq\End_{\Sym_n}(\cP_n).
\end{align}

This algebra is usually described in terms of generators and relations:
\begin{defn}
The nilHecke algebra $\nh_n$ is the $\Bbbk$-algebra with generators $T_1,\dotsc,T_{n-1}$ and $x_1,\dotsc,x_n$, subject to the  relations
\begin{gather}\label{eq:nh1}
	x_ix_j=x_jx_i,\quad T_i^2=0,\quad T_iT_j=T_jT_i \ \ \text{ if } \ |i-j|>1, \\[1ex]\label{eq:nh2}
	T_iT_{i+1}T_i=T_{i+1}T_iT_{i+1},\quad  x_iT_j=T_jx_i \ \ \text{if} \ i\neq j, j+1, \\[1ex] \label{eq:nh3}
	x_iT_i-T_ix_{i+1} =1,\quad T_ix_i-x_{i+1}T_i=1.
\end{gather}
\end{defn}
As $\Sym_n$-linear morphisms of $\cP_n$ each generator $x_i$ act as multiplication by the monomial $x_i$ and each generator $T_j$ acts as the divided difference operator
\[
T_j=\frac{1-s_j}{x_j-x_{j+1}},
\]
where $s_j\in \Sy_n$ acts on $\Pol_n$ by switching $x_j$ and $x_{j+1}$.
It follows that the nilHecke algebra is $\bZ$-graded with $\deg(x_i)=2$ and $\deg(T_j)=-2$.
The generators of the $\nh_n$ are usually depicted with string diagrams
\begin{align}
1=&
\tikzdiagh{3}{
	\draw (-.25,.5) -- (-.25,-.5) node[below] {\plusspacing \scriptsize $1$};
	\draw (0.25,.5) -- (0.25,-.5) node[below] {\plusspacing \scriptsize $2$};
	\node at (0.75,0) {$\dots$};
	\draw (1.25,.5) -- (1.25,-.5) node[below] {\plusspacing \scriptsize $n$};
}\\
T_i=&
\tikzdiagh{3}{
	\draw (-.25,.5) -- (-.25,-.5) node[below] {\plusspacing \scriptsize $1$};
	\draw (0.25,.5) -- (0.25,-.5) node[below] {\plusspacing \scriptsize $2$};
	\node at (0.75,0) {$\dots$};
	\draw (1.25,-.5) node[below] {\plusspacing \scriptsize $i$} .. controls (1.25,0) and (1.75,0) .. (1.75,+.5);
	\draw (1.25,.5) .. controls (1.25,0) and (1.75,0) .. (1.75,-.5) node[below] {\plusspacing \scriptsize $i{+}1$};
	\node at (2.25,0) {$\dots$};
	\draw (2.75,.5) -- (2.75,-.5) node[below] {\plusspacing \scriptsize $n$};
}\\
x_i=&
\tikzdiagh{3}{
	\draw (-.25,.5) -- (-.25,-.5) node[below] {\plusspacing \scriptsize $1$};
	\draw (0.25,.5) -- (0.25,-.5) node[below] {\plusspacing \scriptsize $2$};
	\node at (0.75,0) {$\dots$};
    \draw (1.25,.5) -- (1.25,-.5) node[below] {\plusspacing \scriptsize $i$} node[midway, tikzdot]{};
	\node at (1.75,0) {$\dots$};
	\draw (2.25,.5) -- (2.25,-.5) node[below] {\plusspacing \scriptsize $n$};
} 
\end{align}
The product $xy$ of $x,y \in \nh_n$ is given by stacking the diagram of $x$ atop the one of $y$:
\[
\tikzdiagh{-3}{
	\draw (0,.5) -- (0,-.5);
	\draw (.5,.5) -- (.5,-.5);
	\node at (1, .4) {$\dots$};
	\node at (1, -.4) {$\dots$};
	\draw (1.5,.5) -- (1.5,-.5);
	\filldraw [fill=white, draw=black] (-.25,-.25) rectangle (1.75,.25) node[midway] {$x$};
}
\cdot
\tikzdiagh{-3}{
	\draw (0,.5) -- (0,-.5);
	\draw (.5,.5) -- (.5,-.5);
	\node at (1, .4) {$\dots$};
	\node at (1, -.4) {$\dots$};
	\draw (1.5,.5) -- (1.5,-.5);
	\filldraw [fill=white, draw=black] (-.25,-.25) rectangle (1.75,.25) node[midway] {$y$};
}
=
\tikzdiagh{-3}{
	\draw (0,1) -- (0,-1);
	\draw (0.5,1) -- (0.5,-1);
	\draw (1.5,1) -- (1.5,-1);
	\filldraw [fill=white, draw=black] (-.25,.25) rectangle (1.75,.75) node[midway] {$x$};
	\filldraw [fill=white, draw=black] (-.25,-.75) rectangle (1.75,-.25) node[midway] {$y$};
	\node at (1, .9) {$\dots$};
	\node at (1, -.9) {$\dots$};
	\node at (1, 0) {$\dots$};
}
\]
For example the relation $x_iT_i-T_ix_{i+1}=1$ reads 
\[
\tikzdiagh{5}{
	\draw (1.25,-.5) node[below] {\plusspacing \scriptsize $i$} .. controls (1.25,0) and (1.75,0) .. (1.75,+.5);
	\draw (1.25,.5) .. controls (1.25,0) and (1.75,0) .. (1.75,-.5) node[near start, tikzdot]{} node[below] {\plusspacing \scriptsize $i{+}1$};
}
-
\tikzdiagh{5}{
	\draw (1.25,-.5) node[below] {\plusspacing \scriptsize $i$} .. controls (1.25,0) and (1.75,0) .. (1.75,+.5);
	\draw (1.25,.5) .. controls (1.25,0) and (1.75,0) .. (1.75,-.5) node[near end, tikzdot]{}  node[below] {\plusspacing \scriptsize $i{+}1$};
}=
\tikzdiagh{5}{
	\draw (0,.5) -- (0,-.5) node[below] {\plusspacing \scriptsize $i$};
	\draw (0.5,.5) -- (0.5,-.5) node[below] {\plusspacing \scriptsize $i{+}1$};
}
\]
Given $w\in \Sy_n$ with reduced expression $w=s_{i_1}s_{i_2}\dots s_{i_k}$, we define
\[
T_w\coloneq T_{i_1}T_{i_2}\dots T_{i_k} ,
\]
which by Matsumoto's theorem is independent of the reduced expression chosen for $w$.
For $w_0\in \Sy_n$ the longest element, there exists a primitive indecomposable idempotent 
$$\epsilon_n=(-1)^{\frac{n(n-1)}{2}}T_{w_0}x_1^0x_{2}^{1}\cdots x_n^{n-1},$$ 
which diagrammatically is depicted as
\[
\epsilon_n=(-1)^{\frac{n(n-1)}{2}}
\tikzdiagh{0}{
	\draw (0,1.75) .. controls (0,0) and (5,0) .. (5,-1.75);
	\draw (0,.75) .. controls (0,-.25) and (4,-.5) .. (4,-1.75);
	\draw (5,1.75) .. controls (5,0) and (0,0) .. (0,-1.75);
	\draw (0,.75) .. controls (0,1.25) and (1,1.25) .. (1,1.75);
	\draw (3,-1.75) .. controls (3,-.75) and (0,-.75) .. (0,.25);
	\draw (0,.25) .. controls (0,.75) and (2,.75) .. (2,1.75);
	\draw node[xshift=4.98cm,yshift=-1.55cm, tikzdot]{} node[midway, xshift=4.5cm,yshift=-1.7cm]{\tiny $n{-}1$};
	\draw node[xshift=3.95cm,yshift=-1.55cm, tikzdot]{} node[midway, xshift=3.5cm,yshift=-1.7cm]{\tiny $n{-}2$};
	\draw node[xshift=2.95cm,yshift=-1.55cm, tikzdot]{} node[midway, xshift=2.5cm,yshift=-1.7cm]{\tiny $n{-}3$};
}
\]
As in~\cite[\S6.2]{KhQiSussan} we define, for $\mathbf{i}=(i_1,\dotsc ,i_k)\in \mathbb{N}^k$, the idempotent
\[
\epsilon_{\mathbf{i}}\coloneq \epsilon_{i_1}\otimes\cdots \otimes \epsilon_{i_k}\in \nh_{i_1+i_2+\cdots +i_k}.
\]
The ring $\Pol_n$ can be endowed with a derivation defined by $\dif(x_i)=x_i^2$ for all $i=1,\dotsc,n$, which preserves the subring of symmetric polynomials. The derivation $\dif$ is $p$-nilpotent when the ground field is specialized to $\kp$ and in this case $(\Pol_n,\dif)$ and $(\Sym_n,\dif)$ become $p$-dg algebras over $\kp$. 
For $\alpha=(\alpha_1,\dotsc,\alpha_n)\in \Bbbk^n$, the $\Bbbk$-linear map
\[
\dif_{\alpha} (v)= \sum_{i=1}^n \alpha_i x_iv 
\]
together with the Leibniz rule, is a derivation on $\cP_n$. When we specialise $\alpha$ to $\bF_p^n$ the derivation $\dif_\alpha$ becomes a $p$-differential and $(\cP_n, \dif _{\alpha})$ a $p$-dg module over the $p$-dg algebra $(\Pol_n,\dif)$, or $(\Sym_n,\dif)$. As explained in~\cite[\S$5$]{Qi} there is a natural way to induce a $p$-dg structure on the morphism space of two $p$-dg modules. This allows us to define a differential on the $\kp$-algebra $\nh_n$, depending on $\alpha$.
The choice $\alpha_{i+1}-\alpha_{i}=a\in \bF_p$ makes the differential local, which implies that the canonical inclusion of algebras $\nh_n \hookrightarrow \nh_{n+1}$ is a morphism of $p$-dg algebras. With this choice we write the new differential as $\dif_a\colon \nh_n \to \nh_n$. Its action on the generators is given below:
\begin{equation}\label{eq:pdgx}
\dif_a \left(
	\tikzdiag{
		\draw (0,.5) -- (0,-.5) node[midway, tikzdot]{};	 
	 }
\right)=
\tikzdiag{
	\draw (0,.5) -- (0,-.5) node[near end, tikzdot]{} node[near start, tikzdot]{};
}
=
\tikzdiag{
	\draw (0,.5) -- (0,-.5) node[midway, tikzdot]{}  node[midway, xshift=1.5ex,yshift=0.75ex]{\scriptsize $2$};
}
\end{equation}
\begin{equation}\label{eq:pdgdemazure}
\dif_a \left(
	\tikzdiag{
		\draw (0,-.5) .. controls (0,0) and (.5,0) .. (.5,+.5);
		\draw (0,.5) .. controls (0,0) and (.5,0) .. (.5,-.5);	
	}
\right)=
a \ 
\tikzdiag{
	\draw (0,.5) -- (0,-.5);
	\draw (0.5,.5) -- (.5,-.5);
}
-(a+1) \ 
\tikzdiag{
	\draw (0,-.5) .. controls (0,0) and (.5,0) .. (.5,+.5);
	\draw (0,.5) .. controls (0,0) and (.5,0) .. (.5,-.5) node[near start, tikzdot]{};	
}
+(a-1) \ 
\tikzdiag{
	\draw (0,-.5) .. controls (0,0) and (.5,0) .. (.5,+.5) node[near end, tikzdot]{};
	\draw (0,.5) .. controls (0,0) and (.5,0) .. (.5,-.5);	
}
\end{equation}
As explained in~\cite[\S$4$]{Qi}~\cite[\S$2$]{KhQi} it is possible to define the compact derived category $\mathcal{D}^c(A)$ over a $p$-dg algebra $A$ and its Grothendieck group $\mathbf{K}_0(\mathcal{D}^c(A))=\mathbf{K}_0(A)$.
For the $p$-dg algebra $(\nh_n,\dif_a)$, when $a\in\{\pm 1\}$ and $n<p$, its Grothendieck group is equal to the ring of cyclotomic integers~\cite[Corollary 3.27]{KhQi},
\[
\mathbf{K}_0(\nh_n,\dif_{\pm 1})\cong \mathbb{O}_p.
\]
Most of the time we will work with $a=1$ as in~\cite{EliasQi2} (see remark right after Proposition $3.18$ in~\cite{KhQi}) and omit the subscript on the differential symbol $\dif$ when no confusion arises. With this choice  we have
\begin{equation}\label{eq:dTi}
\dif \left(
	\tikzdiag{
		\draw (0,-.5) .. controls (0,0) and (.5,0) .. (.5,+.5);
		\draw (0,.5) .. controls (0,0) and (.5,0) .. (.5,-.5);	
	}
\right)=
\ 
\tikzdiag{
	\draw (0,.5) -- (0,-.5);
	\draw (0.5,.5) -- (.5,-.5);
}
-2\ 
\tikzdiag{
	\draw (0,-.5) .. controls (0,0) and (.5,0) .. (.5,+.5);
	\draw (0,.5) .. controls (0,0) and (.5,0) .. (.5,-.5) node[near start, tikzdot]{};	
}=-
\tikzdiag{
	\draw (0,-.5) .. controls (0,0) and (.5,0) .. (.5,+.5);
		\draw (0.5,-.5) .. controls (0.5,0) and (0,0) .. (0,.5) node[near end, tikzdot]{};
}
-
\tikzdiag{
	\draw (0,-.5) .. controls (0,0) and (.5,0) .. (.5,+.5);
	\draw (0,.5) .. controls (0,0) and (.5,0) .. (.5,-.5) node[near end, tikzdot]{};	
}
\end{equation}
Using Equation $2.4$ in~\cite[\S2.3]{EliasQi2} it can be proved that 
\begin{equation}\label{eq:e_n}
\dif(\epsilon_n)=-\sum_{i=1}^n(n-i)x_i\epsilon_n.
\end{equation}

\subsection{The enhanced nilHecke algebra} 
One of the main ingredients in the program of categorification of Verma modules 
is the algebra $A_n$ which we describe below (see~\cite[\S2]{naissevaz2} for further details). 
Let $\bV(\und{\omega}_n) \coloneq \bV \left(\omega_{1},\dotsc, \omega_{n}\right)$ be the exterior algebra in variables $\omega_1,\dotsc,\omega_n$. The ring of extended polynomials is defined as $R_n \coloneq \Pol_n\otimes \bV(\und{\omega}_n)$. 
It has the structure of a $\bZ^2$ graded super ring, where the $x_i$ are even and the $\omega_i$ are odd. The $\bZ^2$-grading is defined by declaring $\deg(x_i)=(2,0)$ and $\deg(\omega_i)=(-2i,2)$. For an element $f$ of $R_n$ with degree $\deg(f)=(s,t)$, we call the two values $s$ and $t$ the \emph{$q$-grading} and the \emph{$\lambda$-grading} respectively. To emphasize just one of the two gradings we write either $\deg_q(f)=s$ or $\deg_\lambda(f)=t$.
The symmetric group $\Sy_n$ acts on $R_n$ via the permutation action on $\Pol_n$, while 
\begin{equation} \label{eq:symact}
	s_j(\omega_k)=\omega_k+\delta_{j,k}(x_k-x_{k+1})\omega_{k+1},
\end{equation}
where $s_j$ is the transposition exchanging $x_j$ with $x_{j+1}$ and $\delta_{j,k}$ is the Kronecker delta. 
This action preserves all gradings.
We write $R_n^{\Sy_n}$ for the subalgebra of $\Sy_n$-invariants of $R_n$. The algebra $R_n$ is a $(R_n^{\Sy_n},R_n^{\Sy_n})$-bimodule with basis $U_n \coloneq \Bbbk\left\langle x_1^{b_1}x_2^{b_2}\cdots x_n^{b_n} \ | \ 0\leq b_i \leq n-i \right\rangle$, which implies that
\begin{equation} \label{eq:Rn}
	R_n \cong R_n^{\Sy_n}\otimes U_n.
\end{equation}

\begin{defn}
The \emph{enhanced nilHecke algebra} is defined as the endomorphism ring of $R_n$ viewed as an $R_n^{\Sy_n}$-module
\[
A_n \coloneq \End_{R_n^{\Sy_n}}(R_n).
\]
\end{defn}
As a direct consequence of the isomorphism~\eqref{eq:Rn} the algebra $A_n$ is isomorphic to the matrix algebra of dimension $(n!)^2$ with $R_n^{\Sy_n}$ coefficients:
\[
A_n \coloneq \End_{R_n^{\Sy_n}}(R_n) \cong \End_{R_n^{\Sy_n}}( R_n^{\Sy_n}\otimes U_n)\cong R_n^{\Sy_n}\otimes \End_{\Bbbk}(U_n)\cong  \mat(n!, R_n^{\Sy_n}).
\]
\begin{prop} \label{defn:An}
  The algebra $A_n$ is generated by the elements $x_1,\dotsc,x_n$, $T_1,\dotsc,T_{n-1}$, and $\omega_1,\dotsc,\omega_n$, modulo the nilHecke relations
\begin{gather*}
	x_ix_j=x_jx_i,\qquad T_i^2=0,\qquad T_iT_j=T_jT_i \ \ \text{ if } \ |i-j|>1, \\[1ex]
	T_iT_{i+1}T_i=T_{i+1}T_iT_{i+1}, \qquad x_iT_j=T_jx_i \ \ \text{if} \ i\neq j, j+1, \\[1ex]
	x_iT_i-T_ix_{i+1} =1,\qquad T_ix_i-x_{i+1}T_i=1 ,
\end{gather*}
and the relations involving the $\omega$'s
\begin{gather}
	\label{eq:nhs1}
	\omega_i\omega_j=-\omega_j\omega_i, \\[1ex]
	\label{eq:nhs2}
	x_i\omega_j=\omega_jx_i,\qquad T_i\omega_j=\omega_jT_i \ \ \text{ if } \ i\neq j, \\[1ex]
	\label{eq:nhs3}
	T_i(\omega_i-x_{i+1}\omega_{i+1})=(\omega_i-x_{i+1}\omega_{i+1})T_i
\end{gather}
It is a $\bZ^2$-graded super algebra, where the $x_i$ and the $T_i$ are even, while the $\omega_i$ are odd and the $\bZ^2$ grading is given by $\deg(x_i)=(2,0)$, $\deg(T_j)=(-2,0)$, $\deg(w_i)=(-2i,2)$.
\end{prop}
We call the two gradings $(s,t)$ the $q$-grading and $\lambda$-grading, respectively. 
\begin{proof}
This is a summary of~\cite[Definition $2.3$, Proposition $2.4$, Corollary $2.32$]{naissevaz2}.
\end{proof}
There is a canonical inclusion $\nh_n\hookrightarrow A_n$ which makes $\nh_n$ a subalgebra of $A_n$ concentrated in even parity and trivial $\lambda$-degree.
The algebra $A_n$ is diagrammatically presented with the same diagrams used for the $\nh_n$ with the addition of \textit{floating dots} corresponding to the $\omega_i$:
\[
\omega_i = 
\tikzdiagh{5}{
	\draw (0,.5) -- (0,-.5);
	\node at (.5,0) {$\dots$};
	\draw (1,.5) -- (1,-.5) node[below] {\plusspacing \scriptsize $i$};
	\fdot{}{1.5,0};
	\draw (2,.5) -- (2,-.5) node[below] {\plusspacing \scriptsize $i{+}1$};
	\node at (2.5,0) {$\dots$};
	\draw (3,.5) -- (3,-.5);
}
\]
The relation~\eqref{eq:nhs1} is presented diagrammatically by
\[
\tikzdiagh{-3}{
	\node at (.5,0) {$\dots$};
	\draw (1,.5) -- (1,-.5);
	\fdot{}{1.5,.25};
	\draw (2,.5) -- (2,-.5);
	\node at (2.5,0) {$\dots$};
	\draw (3,.5) -- (3,-.5);
	\fdot{}{3.5,-.25};
	\draw (4,.5) -- (4,-.5);
	\node at (4.5,0) {$\dots$};
} =
-\tikzdiagh{-3}{
	\node at (.5,0) {$\dots$};
	\draw (1,.5) -- (1,-.5);
	\fdot{}{1.5,-.25};
	\draw (2,.5) -- (2,-.5);
	\node at (2.5,0) {$\dots$};
	\draw (3,.5) -- (3,-.5);
	\fdot{}{3.5,.25};
	\draw (4,.5) -- (4,-.5);
	\node at (4.5,0) {$\dots$};
}
\]
This diagram makes clear that the height of the floating dots is important. As a consequence of
\[
\omega_{i+1}=T_i\omega_iT_ix_{i+1}-x_iT_i\omega_iT_i ,
\]
which has the diagrammatic form 
\[
\tikzdiagh{-2}{
	\draw (0,1) -- (0,-1); 
	\draw (.5,1) -- (.5,-1);
	\fdot{}{1,0};
}=
\tikzdiagh{-2}{
	\draw (0,1) .. controls (0,.5) and (1,.5) .. (1,0) .. controls (1,-.5) and (0,-.5) .. (0,-1);
	\draw (1,1) .. controls (1,.5) and (0,.5) .. (0,0) .. controls (0,-.5) and (1,-.5) .. (1,-1) node[near end, tikzdot]{};
	\fdot{}{.5,0};
}
-
\tikzdiagh{-2}{
	\draw (0,1) .. controls (0,.5) and (1,.5) .. (1,0) node[near start, tikzdot]{} .. controls (1,-.5) and (0,-.5) .. (0,-1);
	\draw (1,1) .. controls (1,.5) and (0,.5) .. (0,0) .. controls (0,-.5) and (1,-.5) .. (1,-1);
	\fdot{}{.5,0};
}
\]
all the $\omega_i$ can be expressed using just $\omega_1$.
This results in a presentation of $A_n$ using only one odd generator.
\begin{prop}
The algebra $A_n$ has a presentation by even generators $T_1,\dotsc,T_{n-1}$, and $x_1,\dotsc,x_n$ and the odd generator $\theta$ where the $T_i$ and the $x_i$ satisfy the relations of the nilHecke algebra while 

\begin{gather*}
\theta^2 = 0,    
\\[1ex] 
\theta x_i = x_i\theta\mspace{20mu} \text{for all }i,
\\[1ex] 
\theta T_i = T_i\theta\mspace{20mu} \text{for }i>1,
\\[1ex] 
\theta T_1 \theta T_1 + T_1\theta T_1\theta = 0 .
\end{gather*}

\end{prop}

The action of $A_n$ on the ring of extended polynomials $R_n$ extends the action of $\nh_n$ on the ring of polynomials. In particular for $x_i, \omega_i, T_j \in A_n$, $s_j \in \Sy_n$, $f\in R_n$ we have that:
\begin{gather*}
	x_i(f)=x_if,\\[1ex]
	\omega_i(f)=\omega_if, \\[1ex]
	T_j(f)=\frac{f-s_j(f)}{x_j-x_{j+1}}.
\end{gather*}

In~\cite[\S$2.4$]{naissevaz2} a mild generalization of the $\omega_i$, called \emph{labelled $\omega$} was introduced. 
They will be useful later, when we will study one particular $p$-dg structure on $R_n$.
For $j<k$ positive integers and $\ell$ an integer we denote by $\hs_\ell(j,k)$ the complete homogeneous symmetric polynomial of degree $\ell$ in the variables $x_j, x_{j+1}, \dotsc, x_k$, with the convention that $\hs_0(j,k)=1$ and $\hs_\ell(j,k)=0$ whenever $\ell<0$.
\begin{defn} \label{defn:labelomega}
For $1\leq a$ and $k \leq n$ we define $\omega_k^a$ to be the element of $R_n$ (and therefore of $A_n$) given by
\[
\omega_k^a\coloneq \sum_{\ell=1}^{k}(-1)^{a+k+\ell}\hs_{a+l-k}(\ell,k)\omega_{\ell} .
\]
\end{defn}
Our interest in labelled $\omega$'s lies in the proposition below (this is~\cite[Corollary 2.29]{naissevaz2}), that will be helpful to compute the Grothendieck group of $R_n$ as a $p$-dg algebra. 
To simplify notation we define the symbol $\bV(\und{\omega}_n^a)\coloneq \bV(\omega_n^0,\cdots \omega_n^{n-1})$.

\begin{prop} \label{prop:Rnsym}
There is an isomorphism of $\Bbbk$-algebras
\[
R_n^{\Sy_n}\cong \Pol_n^{\Sy_n}\otimes \bV(\und{\omega}_n^a) .
\]
\end{prop}
It follows from~\autoref{prop:Rnsym} that as $R_n^{\Sy_n}$-modules 
\begin{equation} \label{eq:R_niso}
	R_n\cong R_n^{\Sy_n}\otimes U_n \cong \Pol_n^{\Sy_n}\otimes \bV(\und{\omega}_n^a)\otimes U_n,
\end{equation}
and thus
\[
A_n = \End_{R_n^{\Sy_n}}(R_n) = \End_{R_n^{\Sy_n}}( R_n^{\Sy_n}\otimes U_n)\cong R_n^{\Sy_n}\otimes \End_{\Bbbk}(U_n).
\]
This implies that
\begin{equation} \label{eq:Andecomposition}
A_n \cong \bV (\und{\omega}_n^a)\otimes \Pol_n^{\Sy_n}\otimes \End_{\Bbbk}(U_n),
\end{equation} 
and, since $\nh_n\cong \Pol_n^{\Sy_n}\otimes \End_{\Bbbk}(U_n)$, we obtain an isomorphism 
\begin{equation} \label{eq:AnNhdecomposition}
A_n \cong \bV (\und{\omega}_n^a)\otimes \nh_n.
\end{equation}


%% file: files/pdgAn.tex
%

\subsection{A derivation on \texorpdfstring{$A_n$}{Rn}}
We would like to define $p$-differentials on the ring $R_n$ and on the algebra $A_n$ such that the inclusions $\Pol_n \subset R_n$ and $\nh_n \subset A_n$ become inclusions of $p$-dg algebras, where $\Pol_n$ and $\nh_n$ have Khovanov--Qi's $p$-dg structure given in~\eqref{eq:pdgx} and~\eqref{eq:pdgdemazure} above.
Recall the balanced free left $\Pol_n$ $p$-dg module $\cP_n=\Pol_n.v$ from the previous section.
Recall also that $\deg(v)=\frac{n(1-n)}{2}$ and $\dif_{a}(v)=g_\alpha.v$, with $g_\alpha=\sum_{i=1}^n\alpha_ix_i$ and $\alpha_{i+1}-\alpha_i=a \in \mathbb{F}_p$. 
From the action of $\nh_n$ on $\cP_n$ induced from the one on $\Pol_n$,
where $T_i(f.v)=T_i(f).v$, one induces a
$p$-differential on $\nh_n$ such that $\cP_n$
becomes a $p$-dg module over the $p$-dg algebra $\nh_n$.
We will use this procedure on $R_n$ and $A_n$. We will work over the base field $\Bbbk$, and we first work out derivation that become $p$-differentials once we specialize the parameters and the base field in~\autoref{prop:pnilp}.

First we form the balanced free left $R_n$-module
\[
\cR_n = \cP_n\otimes \bV(\und{\omega}) = R_n.v ,
\]
where $v$ is as above. For $\xi \in A_n=\End_{R^{\Sy_n}}(\cR_n)$ and $f\in \cR_n$, we require the derivation $\dif$ on $A_n$ to satisfy
\[
\dif(\xi )(f)=\dif(\xi (f))-\xi (\dif(f)),
\]
and to restrict to Khovanov--Qi's derivation $\dif_{a}$ on $\nh_n$. These restrictions, together with the fact that $\omega_{i+1} = (-1)^{i}T_{i}\dotsm T_2T_1(\omega_1)$~\cite[\S2.1.1]{naissevaz2}, allow us to compute
\begin{align*}
\dif_{a}(T_i)(\omega_i.v) &= \dif(T_i(\omega_i.v))-T_i(\dif(\omega_i.v))\\
	                         &=-\dif(\omega_{i+1}.v)-T_i(\dif(\omega_i.v)),
\end{align*}
which implies that
\begin{align*}
\dif(\omega_{i+1}.v) &= -\dif_{a}(T_i)(\omega_i.v)-T_i(\dif(\omega_i.v))\\
                    &= -\bigl( (x_i+x_{i+1})\omega_{i+1} + T_i(\dif \omega_i)\bigr).v +\omega_{i+1}g_\alpha.v.
\end{align*} 
Therefore, since $\dif(\omega_{i+1}.v)=\dif(\omega_{i+1}).v+\omega_{i+1}g_\alpha .v$, the action of the derivation $\dif$ on the exterior part of $R_n$ can be defined as
\begin{equation} \label{eq:pdiff1}
\dif(\omega_{i+1}) = - (x_i+x_{i+1})\omega_{i+1} - T_i(\dif \omega_i) .
\end{equation}
Note that $\dif(\omega_{i+1})$ does not depend on $a$ and it is uniquely determined by the $q$-degree zero element $\dif(\omega_1)\in R_n$.

\begin{rem} \label{rem:pdiff1}
    Equation \eqref{eq:pdiff1} implies that $s_i\dif(\omega_i)=\dif(s_i\omega_i)$, indeed  
    \begin{align*}
    \dif\left(s_i\omega_i\right)&=\dif\left(\omega_i+(x_i-x_{i+1})\omega_{i+1}\right)=\dif\left(\omega_i \right)+(x_i^2-x_{i+1}^2 )\omega_{i+1}+(x_i-x_{i+1})\dif\left( \omega_{i+1}\right)\\
    &=\dif\left(\omega_i \right) -(x_i-x_{i+1})T_i(\dif \left( \omega_i \right))
    =\dif\left(\omega_i \right) -(x_i-x_{i+1})\frac{\dif(\omega_i)-s_i \dif(\omega_i)}{x_i-x_{i+1}} 
    =s_i\dif\left(\omega_i \right).
\end{align*}
\end{rem}

\begin{prop}
    Suppose that $s_j\dif\left(\omega_1 \right)=\dif\left(s_j\omega_1 \right)$ holds for $2\leq j \leq n-1$. 
    Then the derivation $\dif$ is $\Sy_n$-equivariant. 
\end{prop}
By $\Sy_n$-equivariant we mean that $s_j\dif\left(\omega_{i} \right)=\dif\left(s_j\omega_{i} \right)$ for any $1\leq i \leq n$ and $1\leq j \leq n-1$.
\begin{proof}
When $j\neq i$ we have that $s_j\omega_i=\omega_i$, and we just need to prove that $s_j\dif(\omega_i)=\dif(\omega_i)$.
When $i=1$ \autoref{rem:pdiff1} assures that $s_1\dif(\omega_1)=\dif(s_1\omega_1)$.
When $i=2$ the thesis is trivial for $j=1$, since from~\eqref{eq:pdiff1} it is clear that $\dif(\omega_2)$ is $s_1$ invariant. For $j=2$ we use again~\autoref{rem:pdiff1}. While, for $j\geq 3$ we use the hypothesis $s_j\dif(\omega_1)=\dif(s_j\omega_1)=\dif(\omega_1)$: 
\begin{align*}
    s_j\dif(\omega_2)&=-(x_1+x_2)\omega_2-s_jT_1(\dif(\omega_1))\\
        &=-(x_1+x_2)\omega_2-T_1(s_j\dif(\omega_1))\\
        &=-(x_1+x_2)\omega_2-T_1(\dif(\omega_1))=\dif(\omega_2)=\dif(s_j\omega_2).
\end{align*}
Now we consider $i$ arbitrary, when $j\leq i-3$ the induction step is proved easily using~\eqref{eq:pdiff1} and the induction hypothesis, as we just did for $i=2$ and $j\geq 3$. For $j=i-2$, applying twice~\eqref{eq:pdiff1} we obtain the equation
\begin{align*}
    \dif(\omega_{i})=-(x_{i-2}+x_{i-1}+x_i)\omega_i + \omega_{i-1}+ T_{i-1}T_{i-2}\left(\dif(\omega_{i-2})\right).
\end{align*}
Thanks to the induction hypothesis we know that $s_{i-1}\dif(\omega_{i-2})=\dif(s_{i-1}\omega_{i-2})$, which implies that $\dif(\omega_{i-2})$ is $s_{i-1}$ invariant, and thus $T_{i-1}T_{i-2}\left(\dif(\omega_{i-2})\right)$ is $s_{i-2}$ invariant. So, $\dif(\omega_i)$ is $s_{i-2}$ invariant.
The remaining cases are immediate. For $j=i-1$ it is clear from~\eqref{eq:pdiff1} that $s_{i-1}\dif(\omega_i)=\dif(\omega_i)$. For $j=i$ we use~\autoref{rem:pdiff1}, and for $j\geq i+1$ the computation is the same that we did for $i=2$ and $j\geq3$.
\end{proof}

\begin{rem} \label{rel:diszero}
One might wonder why not define the derivation on $R_n$ to act trivially on the exterior part and having $\dif(x_i)=x_i^2$ for $1\leq i \leq n$. The reason against it is that when $n>1$ this derivation does not commute with the action of the symmetric group $\Sy_n$, and thus it does not restrict to a derivation on $R_n^{\Sy_n}$, which is necessary to induce a derivation on the algebra $A_n=\End_{R_n^{\Sy_n}}(\cR_n)$. 
\end{rem}

\begin{prop} \label{prop:kpowers}
For each $k\geq 1$ and $1 \leq i \leq n-1$ we have
\begin{equation*}
\dif^k\omega_{i+1}=-k(k-1)x_ix_{i+1}T_i(\dif^{k-2}\omega_i)+k(x_i+x_{i+1})T_i(\dif^{k-1}\omega_i)-T_i(\dif^k\omega_i),
\end{equation*}
where $\dif^{-1}$ and $\dif^0=1$ act as the identity by convention.
\end{prop}

\begin{proof}
Since $(R_n, \dif)$ is isomorphic to $(\cR_n=R_n.v,\dif')$ as a module over itself, where $\dif' v=0$, we will prove the proposition for $\cR_n$. In this setting,~\eqref{eq:pdgdemazure} together with $\dif'(T_i)(f.v)=\dif'(T_i(f).v)-T_i(\dif'(f.v))$ give us the equality
    \begin{equation} \label{eq:dt}
        \dif'(T_i(f.v))=-(x_i+x_{i+1})T_i(f.v)+T_i(\dif' (f.v)).
    \end{equation}
Using this result and the fact that $T_i(-\omega_i.v)=\omega_{i+1}.v$ we can check the statement for $k=1$,
	\begin{align*}
		&\dif'\omega_{i+1}.v=(x_i+x_{i+1})T_i(\omega_i.v)-T_i(\dif'\omega_i.v).
	\end{align*}
Proving the inductive step by computing $\dif'\dif'^{k-1}(\omega_{i+1})$ using~\eqref{eq:dt} is straightforward, and it gives
\begin{equation*}
		\dif'^k\omega_{i+1}.v=-k(k-1)x_ix_{i+1}T_i(\dif'^{k-2}\omega_i.v)+k(x_i+x_{i+1})T_i(\dif'^{k-1}\omega_i.v)-T_i(\dif'^k\omega_i.v).
\qedhere
 \end{equation*}
\end{proof}

\smallskip

Our next task is to find a derivation $\dif_n$ on $R_n$ which satisfies~\eqref{eq:pdiff1}, it is $\Sy_n$-equivariant, and has the following form:
    \begin{equation} \label{eq:dif_n}
        \begin{split}
         \dif_n(x_i) &= x_i^2 \mspace{90mu} \text{if } 1\leq i \leq n , \\
         \dif_n(\omega_i) &= \sum_{j=i+1}^n\alpha_{i,j}\omega_j \mspace{24mu}  \text{if } 1\leq i \leq n-1 , \\
         \dif_n(\omega_n) &= 0 ,
        \end{split}
    \end{equation}
where each $\alpha_{i,j}$ is a homogeneous polynomial of $q$-degree $2j+2$.
    
Note that such $\dif_n$ depends on $n$, it extends the one defined on $\Pol_n$ and it preserves the $\lambda$-grading. This last statement follows directly by the fact that $\dif_n$ preserves the $\lambda$-grading on the generators, and thus on all homogeneous elements, as a consequence of the Leibniz's rule.

\begin{prop}
The polynomials $\alpha_{i,j}$  are such that
\begin{equation}\label{eq:pdiff2}
	T_i(\alpha_{i,j})=
	\begin{cases}
		-(x_{i}+x_{i+1}) \quad & \text{if } j=i+1, \\
		-\alpha_{i+1,j} \quad & \text{if } j>i+1.
	\end{cases}
\end{equation}
\end{prop}
\begin{proof}
We replace the derivations in~\eqref{eq:pdiff1} using the definition of $\dif_n$ in~\eqref{eq:dif_n} to obtain 
\begin{align*}
	\sum_{j=i+2}^n\alpha_{i+2,j}\omega_j=\dif_n\omega_{i+1}&=-\Bigl( (x_i+x_{i+1})\omega_{i+1} + T_i(\dif_n\omega_i)\Bigr)\\
	&=-\Bigl( (x_i+x_{i+1})\omega_{i+1} + T_i\bigl( \sum_{j=i+1}^n \alpha_{i,j} \omega_{j}\bigr) \Bigr)\\
    &=-\Bigl((x_i+x_{i+1})\omega_{i+1} +\sum_{j=i+1}^nT_i(\alpha_{i,j})\omega_j\Bigr).
\end{align*}
The proposition follows as a consequence of the equation
\begin{equation*} 
    \sum_{j=i+2}^n\alpha_{i+2,j}\omega_j=-\Bigl((x_i+x_{i+1})\omega_{i+1} +\sum_{j=i+1}^nT_i(\alpha_{i,j})\omega_j\Bigr).
\qedhere
\end{equation*}
\end{proof}

We require $\dif_n$ to be $\Sy_n$-equivariant because later on we need it to restrict to a $p$-differential on the subalgebra of $\Sy_n$-invariants of $R_n$. Equating $s_l\dif_n(\omega_i)$ with $\dif_n(s_l\omega_i)$, for $s_l\in\Sy_n$ the transposition $(l,l+1)$, we obtain that
\begin{align} 
\label{eq:pdiff3}
	T_{j}(\alpha_{i,j+1}) &= \alpha_{i,j} && \text{for }  i+1\leq j \leq n-1  \\
\label{eq:pdiff32}
    s_l(\alpha_{i,j+1}) &= \alpha_{i,j+1} && \text{for } l\neq j,i \text{ and } i+1\leq j \leq n-1 .
\end{align}
This and the previous proposition can be summarized in the diagram below.
\begin{center}
\begin{tikzpicture}
	\matrix[matrix of math nodes,column sep={80pt,between origins},row
    sep={50pt,between origins},nodes={asymmetrical rectangle}] (s)
  {
  &|[name=11]| \vdots  &|[name=12]| \vdots &|[name=13]|\vdots &|[name=14]| &|[name=15]| \vdots\\
  &|[name=21]| (x_{i-1}+x_i)  &|[name=22]|\alpha_{i,i+1}  &|[name=23]|\alpha_{i,i+2} &|[name=24]| \cdots & |[name=25]| \alpha_{i,n} \\
  &|[name=31]|   &|[name=32]|(x_{i}+x_{i+1})  &|[name=33]|\alpha_{i+1,i+2} &|[name=34]| \cdots & |[name=35]| \alpha_{i+1,n} \\
  &|[name=41]|  &|[name=42]|  &|[name=43]| \vdots&|[name=44]| &|[name=45]|\vdots \\
  &|[name=51]|  &|[name=52]|  &|[name=53]| &|[name=54]| &|[name=55]| \alpha_{n-1,n} \\ 
  &|[name=61]|  &|[name=62]|  &|[name=63]| &|[name=64]| &|[name=65]|(x_{n-1}+x_n) \\   
  };
  \draw[->] 
  (11)edge node[auto, left] {\small$-T_{i-1}$} (21)
  (12)edge node[auto, left] {\small$-T_{i-1}$} (22)
  (13)edge node[auto, left] {\small$-T_{i-1}$} (23)
  (15)edge node[auto, left] {\small$-T_{i-1}$} (25)
  (22)edge node[auto, left] {\small$-T_{i}$} (32)
  (23)edge node[auto, left] {\small$-T_{i}$} (33)
  (25)edge node[auto, left] {\small$-T_{i}$} (35)
  (23)edge node[auto, above] {\small$T_{i+1}$} (22)
  (24)edge node[auto, above] {\small$T_{i+2}$} (23)
  (25)edge node[auto, above] {\small$T_{n-1}$} (24)
  (34)edge node[auto, above] {\small$T_{i+2}$} (33)
  (35)edge node[auto, above] {\small$T_{n-1}$} (34)
  (33)edge node[auto, left] {\small$-T_{i+1}$} (43)
  (35)edge node[auto, left] {\small$-T_{i+1}$} (45)
  (45)edge node[auto, left] {\small$-T_{n-2}$} (55)
  (55)edge node[auto, left] {\small$-T_{n-1}$} (65)
  ;
\end{tikzpicture}
\end{center}
The action of any other $T_j$ which does not appear in the picture is zero. Moreover, all facets in the diagram commute. It becomes clear that once we have defined $\alpha_{1,n}$ we can use the Demazure operators to obtain all the other $\alpha_{i,j}$ for $1\leq i \leq n-1$ and $2\leq j \leq n$.
\begin{defn} \label{def:alpha_ij}
    We define the polynomials
    \[
        \alpha_{i,j}=\sum_{l=j}^n x_l^2\prod_{k=i+1}^{j-1}(x_{k}-x_l) \mspace{40mu} \text{for} \quad 1\leq i < j \leq n,
    \]
    where we use the convention that $\prod_{k=i+1}^{i}(x_k-x_l)=1$.
\end{defn}
\begin{prop} \label{prop:pdiff}
The polynomials $\alpha_{i,j}$ satisfy the relations given in~\eqref{eq:pdiff2},~\eqref{eq:pdiff3} and~\eqref{eq:pdiff32}.\end{prop}
Before proving this proposition we prepare a lemma:
\begin{lem} \label{lem:pdiff}
For $1\leq i \leq j-1 < j\leq n$ we have
\begin{align*}
	T_{j-1}\Bigl(x_j^2\prod_{k={i+1}}^{j-1}(x_k-x_j) \Bigr)=x_j^2\prod\limits_{k=i+1}^{j-2}(x_k-x_j)+x_{j-1}^2\prod\limits_{k=i+1}^{j-2}(x_k-x_{j-1}) .
\end{align*}
\end{lem}

\begin{proof}
    We prove the lemma by reverse induction on $i$, making it decrease to 1 starting from $j-2$. 
 As base case we fix $i=j-2$ where
    \[
T_{j-1}\biggl(x_j^2\prod_{k={j-1}}^{j-1}(x_k-x_j) \biggr)=T_{j-1}\bigl(x_j^2(x_{j-1}-x_j)\bigr)=x_{j-1}^2 + x_j^2.    
    \]
    
 We now suppose that the claim is true for $i+1$ and we prove that it is true also for $i$:
    \begin{align*}
      T_{j-1} \Bigl( x_j^2\prod_{k={i+1}}^{j-1}(x_k-x_j) \Bigr)
          &=x_j^2\prod _{k=i+2}^ {j-1}(x_k-x_j) \\
        & \quad \quad + (x_{i+1}-x_{j-1})\Bigl(x_j^2\prod_{k=i+2}^{j-2}(x_k-x_j)+x_{j-1}^2\prod_{k=i+2}^{j-2}(x_k-x_{j-1}) \Bigr)\\
        &=(x_{j-1}-x_j)x_j^2\prod _{k=i+2}^ {j-2}(x_k-x_j) \\
        & \quad \quad + (x_{i+1}-x_{j-1})\Bigr(x_j^2\prod_{k=i+2}^{j-2}(x_k-x_j)+x_{j-1}^2\prod_{k=i+2}^{j-2}(x_k-x_{j-1}) \Bigr)\\
        &=x_{j}^{2}\prod _{k=i+1}^{j-2}(x_k-x_j)+x_{j-1}^2\prod_{k=i+1}^{j-2}(x_k-x_{j-1}).
    \end{align*}
    To obtain the first equality we used the Leibniz rule for a Demazure and the induction hypothesis, while the final result is just a rearrangement of the terms using the associativity and distributivity rules.
\end{proof}

\begin{proof}[Proof of~\autoref{prop:pdiff}]
    Equations~\eqref{eq:pdiff2} and~\eqref{eq:pdiff32} are straightforward. For~\eqref{eq:pdiff3} we want to prove that $T_{j-1}(\alpha_{i,j})=\alpha_{i,j-1}$, for $i+1\leq j \leq n-1$. It is also a computation, except for the second to last equality below, where we use~\autoref{lem:pdiff}:
    \begin{align*}
        T_{j-1}(\alpha_{i,j})&=\sum_{l=j}^nT_{j-1}\Bigl(x_l^2\cdot\prod_{k=i+1}^{j-1}(x_{k}-x_l)\Bigr)\\
        &=T_{j-1}\Bigl( x_j^2\prod_{k=i+1}^{j-1}(x_k-x_j) \Bigr)+\sum_{l=j+1}^nx_l^2\cdot\prod_{k=i+1}^{j-2}(x_{k}-x_l) \\
        &=x_{j-1}^2 \prod _{k=i+1}^{j-2}(x_k-x_{j-1})+x_j^2\prod _{k=i+1}^{j-2}(x_k-x_j)+\sum_{l=j+1}^nx_l^2\cdot\prod_{k=i+1}^{j-2}(x_{k}-x_l)\\
        &=\sum_{l=j-1}^{n}x_l \prod_{k=i+1}^ {j-2}(x_k-x_l)=\alpha_{i,j-1}. \qedhere
    \end{align*}
\end{proof}

Denote by $\es_{i}(j,k)$ the elementary homogeneous symmetric polynomial of degree $i$ in the variables $x_j,\dotsc, x_k$. 
\begin{prop} \label{prop:difelem}
We have
\[
\alpha_{i,j}=\sum_{l=j}^{n}x_l^2\sum_{k=0}^{j-i-1}(-1)^{k}\es_{j-i-1-k}(i+1,j-1)x_l^{k} .
\]
\end{prop}

\begin{proof}
	This follows directly from~\autoref{prop:pdiff} together with replacing $t$ by $x_l$ in the well-known identity
\[
	\prod_{k=1}^{n}(x_k-t)=\sum_{k=0}^{n}(-1)^{k}\es_{n-k}(1,n)t^{k} .
\qedhere
\]
\end{proof}

\begin{ex}
Let us consider $R_5$ and compute the action of $\dif_5$ on the $\omega$'s using~\autoref{prop:pdiff}.\\
For $i=1$ we have: 
\begin{align*}
	&\alpha_{1,5}=x_5^2(x_2-x_5)(x_3-x_5)(x_4-x_5),\\
	&\alpha_{1,4}=T_4(\alpha_{1,5})=x_4^2(x_2-x_4)(x_3-x_4)+x_5^2(x_2-x_5)(x_3-x_5),\\
	&\alpha_{1,3}=T_3T_4({\alpha_{1,5}})=x_3^2(x_2-x_3)+x_4^2(x_2-x_4)+x_5^2(x_2-x_5),\\
	&\alpha_{1,2}=T_2T_3T_4(\alpha_{1,5})=x_2^2+x_3^2+x_4^2+x_5^2,\\
\end{align*}
and therefore
\begin{align*}
	\dif_5(\omega_1)=&(x_2^2+x_3^2+x_4^2+x_5^2)\omega_2+\bigl(x_3^2(x_2-x_3)+x_4^2(x_2-x_4)+x_5^2(x_2-x_5)\bigr)\omega_3+\\
	&+\bigl(x_4^2(x_2-x_4)(x_3-x_4)+x_5^2(x_2-x_5)(x_3-x_5)\bigr)\omega_4+\\
	&+x_5^2(x_2-x_5)(x_3-x_5)(x_4-x_5)\omega_5.
\end{align*}
For $i=2$ we have:
\begin{align*}
	&\alpha_{2,5}=-T_1(\alpha_{1,5})=x_5^2(x_3-x_5)(x_4-x_5),\\
	&\alpha_{2,4}=-T_4T_1(\alpha_{1,5})=x_5^2(x_3-x_5)+x_4^2(x_3-x_4),\\
	&\alpha_{2,3}=-T_3T_4T_1(\alpha_{1,5})=x_5^2+x_4^2+x_3^2,\\
\end{align*}
and therefore
\begin{align*}
	\dif_5(\omega_2)=(x_5^2+x_4^2+x_3^2)\omega_3+\bigl(x_5^2(x_3-x_5)+x_4^2(x_3-x_4)\bigr)\omega_4+x_5^2(x_3-x_5)(x_4-x_5)\omega_5.
\end{align*}
For $i=3$ we have:
\begin{align*}
	&\alpha_{3,5}=T_2T_1(\alpha_{1,5})=x_5^2(x_4-x_5),\\
	&\alpha_{3,4}=T_4T_2T_1(\alpha_{1,5})=x_5^2+x_4^2,
\end{align*}
and therefore
\begin{align*}
	\dif_5(\omega_3)=(x_5^2+x_4^2)\omega_4+x_5^2(x_4-x_5)\omega_5.
\end{align*}
For $i=4$ we have:
\begin{align*}
	\alpha_{4,5}=-T_3T_2T_1(\alpha_{1,5})=x_5^2.
\end{align*}
and
\begin{align*}
	\dif_5(\omega_4)=x_5^2\omega_5.
\end{align*}
For $i=5$ we have $\alpha_{5,5}=0$ and $\dif_5(\omega_5)=0$.
\end{ex}

\subsection{Partial derivatives on the extended polynomial ring \texorpdfstring{$R_n$}{}}\label{sec:partder}
Until now we have established the existence of a derivation $\dif_n$ that respects the condition~\eqref{eq:pdiff1}, commutes with the action of the symmetric group, and satisfies~\eqref{eq:dif_n}. In the main result of this section,~\autoref{prop:pnilp}, we prove that $\dif_n$ is $p$-nilpotent when working over a field $\kp$ of characteristic $p>0$. As a step in proving this result we extend to $R_n$ the partial derivatives $\frac{\partial }{\partial x_r}$ on $\Pol_n$.

\begin{defn} \label{defn:partialder}
For each $r\in\{1,\dotsc, n\}$ we define the \emph{partial derivative} of $\omega_i$ with respect to $x_r$ as
  \[
    \dfrac{\partial }{\partial x_r}(\omega_i) =
    \begin{cases}
      0, & \text{ when }i\geq r ,
      \\
      \sum\limits_{j=i+1}^{r}\prod\limits_{k=i+1}^{j-1}(x_k-x_r)\omega_j, & \text{ when } i<r ,
    \end{cases}
  \]
where we use the convention that $\prod_{k=i+1}^{i}(x_k-x_l)=1$.
\end{defn}
This $\Bbbk$-linear operator extends to a derivation on $R_n$ by the Leibniz rule, where $\dfrac{\partial }{\partial x_r}$ acts on polynomials as a partial derivative. Moreover, observe that the action of the partial derivative $\dfrac{\partial}{\partial x_r}$ can be defined on  $R_n$ for $n\geq r$, and it does not depend on $n$.
The following is immediate.
\begin{lem} \label{lem:difaspartial}
    We have
\[
\dif_n(-)=\sum\limits_{r=1}^nx_r^2 \dfrac{\partial }{\partial x_r}(-) \colon R_n \rightarrow R_n.
\]
\end{lem}

\begin{rem}\label{rem:partder-rec}
It follows at once from~\autoref{defn:partialder} that the partial derivative of $\omega_i$ satisfies the recursive relation 
  \begin{equation} \label{eq:partdequation}
 \dfrac{\partial }{\partial x_r}(\omega_i) = \omega_{i+1}+(x_{i+1}-x_r)\dfrac{\partial }{\partial x_r}(\omega_{i+1}) .
  \end{equation}
It can also be computed as
\[
\dfrac{\partial }{\partial x_r}(\omega_i)
= -s_{r-1}\dotsc s_{i+1}T_i(\omega_i) ,
\]
where $j\leq i \leq r$.
\end{rem}

\begin{prop} \label{prop:doublepart}
	We have that
	$$
		\frac{\partial}{\partial x_r}\frac{\partial}{\partial x_r}\omega_i=0.
	$$
\end{prop}
\begin{proof}
	The statement is trivially true if $i\geq {r-1}$, while for $i<r-1$ it is enough to iterate~\eqref{eq:partdequation}: 
	\begin{align*}
		\frac{\partial}{\partial x_r}\frac{\partial}{\partial x_r}\omega_i&=\frac{\partial}{\partial x_r}\biggl(\omega_{i+1}+(x_{i+1}-x_r)\frac{\partial}{\partial x_r}\omega_{i+1}\biggr)=(x_{i+1}-x_r)\frac{\partial}{\partial x_r}\frac{\partial}{\partial x_r}\omega_{i+1}\\
		&=(x_{i+1}-x_r)(x_{i+2}-x_r)\cdots(x_{r-1}-x_r)\frac{\partial}{\partial x_r}\frac{\partial}{\partial x_r}\omega_{r-1} = 0. 
\qedhere
	\end{align*}
\end{proof}

The following justifies the terminology.  
\begin{prop}\label{prop:partialdcommute}
  Partial derivatives commute, i.e. 
    \[
    \dfrac{\partial }{\partial x_r}\dfrac{\partial }{\partial x_s} =
    \dfrac{\partial }{\partial x_s}\dfrac{\partial }{\partial x_r} .
  \]
\end{prop}

\begin{proof}
It is enough to show that 
  \begin{equation}\label{eq:partdcomm}
    \dfrac{\partial }{\partial x_r}\dfrac{\partial }{\partial x_s}(\omega_i) =
    \dfrac{\partial }{\partial x_s}\dfrac{\partial }{\partial x_r}(\omega_i),
  \end{equation}
since the claim is obviously true for the generators $x_i$'s.
The statement is clear for $i\geq s$ or $i\geq r$, for $s=r$ and for $i=n-1,n$ since in each of these cases both sides of~\eqref{eq:partdcomm} are zero. Suppose now w.l.o.g. that $i$ is fixed and

\smallskip
\noindent$\bullet$~$r=i+1$ and $s>r$, we compute
\[
    \dfrac{\partial}{\partial x_{i+1}}\dfrac{\partial}{\partial x_s}\omega_i=\dfrac{\partial}{\partial x_{i+1}}\biggl(\omega_{i+1}+(x_{i+1}-x_s)\dfrac{\partial}{\partial x_s}\omega_{i+1} \biggr)=\dfrac{\partial}{\partial x_s}\omega_{i+1}=\dfrac{\partial}{\partial x_s}\dfrac{\partial}{\partial x_{i+1}}\omega_i.
\]

\smallskip
\noindent$\bullet$~$r=i+2$ and $s>r$, we compute
    \begin{align*}
        \dfrac{\partial}{\partial x_{i+2}}\dfrac{\partial}{\partial x_{s}}\omega_i&=\dfrac{\partial}{\partial x_{i+2}}\biggl(\omega_{i+1}+(x_{i+1}-x_{i+2})\omega_{i+2}+(x_{i+1}-x_s)(x_{i+2}-x_s)\dfrac{\partial}{\partial x_s}\omega_{i+2} \biggr) \\
        &=\omega_{i+2}+(x_{i+1}-x_s)\dfrac{\partial}{\partial x_s}\omega_{i+2}, \\
        \dfrac{\partial}{\partial x_s}\dfrac{\partial}{\partial x_{i+2}}\omega_i&=\dfrac{\partial}{\partial x_s}\bigl(\omega_{i+1}+(x_{i+1}-x_{i+2})\omega_{i+2} \bigr) \\
        &=\omega_{i+2} +(x_{i+2}-x_s)\dfrac{\partial}{\partial x_s}\omega_{i+2} +(x_{i+1}-x_{i+2})\dfrac{\partial}{\partial x_s}\omega_{i+2} \\
        &=\omega_{i+2}+(x_{i+1}-x_s)\dfrac{\partial}{\partial x_s}\omega_{i+2}.
    \end{align*}

\smallskip
\noindent$\bullet$~$r\neq s > i+2$, we proceed by recurrence.  We  first note that since
\begin{align*}
  \dfrac{\partial }{\partial x_{r}}\dfrac{\partial }{\partial x_{s}}\omega_{i+1} &=
  \dfrac{\partial }{\partial x_{r}}\omega_{i+2} + (x_{i+2}-x_s)\dfrac{\partial }{\partial x_{r}}\dfrac{\partial }{\partial x_{s}}\omega_{i+2},
\\
  \dfrac{\partial }{\partial x_{s}}\dfrac{\partial }{\partial x_{r}}\omega_{i+1} &=
  \dfrac{\partial }{\partial x_{s}}\omega_{i+2} + (x_{i+2}-x_r)\dfrac{\partial }{\partial x_{s}}\dfrac{\partial }{\partial x_{r}}\omega_{i+2},
\end{align*}
the recurrence assumption implies that 
\begin{equation} \label{eq:partder-usefuleq}
 \dfrac{\partial }{\partial x_{s}}\omega_{i+2} - \dfrac{\partial }{\partial x_{r}}\omega_{i+2}  = 
(x_r-x_s)\dfrac{\partial }{\partial x_{r}}\dfrac{\partial }{\partial x_{s}}\omega_{i+2} .
\end{equation}
Expanding (we use the recurrence relation once)
\begin{align*}\allowdisplaybreaks
  \dfrac{\partial }{\partial x_{r}}\dfrac{\partial }{\partial x_{s}}\omega_{i} =&
  \omega_{i+2} + (x_{i+2}-x_r)  \dfrac{\partial }{\partial x_{r}}\omega_{i+2}  + (x_{i+1}-x_s)\dfrac{\partial }{\partial x_{s}}\omega_{i+2}
\\ &+(x_{i+1}-x_s)(x_{i+2}-x_r)\dfrac{\partial }{\partial x_{r}}\dfrac{\partial }{\partial x_{s}}\omega_{i+2}
\intertext{and}
  \dfrac{\partial }{\partial x_{s}}\dfrac{\partial }{\partial x_{r}}\omega_{i} =&
  \omega_{i+2} + (x_{i+2}-x_s)  \dfrac{\partial }{\partial x_{s}}\omega_{i+2}  + (x_{i+1}-x_r)\dfrac{\partial }{\partial x_{r}}\omega_{i+2}
\\ &+(x_{i+1}-x_r)(x_{i+2}-x_s)\dfrac{\partial }{\partial x_{r}}\dfrac{\partial }{\partial x_{s}}\omega_{i+2}
\end{align*}
gives
\[
 \dfrac{\partial }{\partial x_{r}}\dfrac{\partial }{\partial x_{s}}\omega_{i}- \dfrac{\partial }{\partial x_{s}}\dfrac{\partial }{\partial x_{r}}\omega_{i}
  =(x_{i+2}-x_{i+1})
  \biggl(  \dfrac{\partial }{\partial x_{r}}\omega_{i+2} - \dfrac{\partial }{\partial x_{r}}\omega_{i+2} +
  (x_r-x_s)  \dfrac{\partial }{\partial x_{r}}\dfrac{\partial }{\partial x_{s}}\omega_{i}\biggr) , 
\]
which is zero by~\eqref{eq:partder-usefuleq}. 
\end{proof}

Now we aim to demonstrate the main result of this section: specializing the characteristic of the ground field results in $p$-nilpotency of $\dif_n$.

\begin{lem} \label{lem:powerx2patial}
    For $n\geq1$ we have the equality
    \[
        \biggl(x^2 \dfrac{\partial}{\partial x} \biggr)^n = \sum_{i=1}^{n}a_{i,n}x^{n+i}\dfrac{\partial^i}{\partial x^i},
    \]
    where
    \begin{align*}
        &a_{1,n} = n!, \\
        &a_{n,n} = 1, \\
        &a_{j,n} = a_{j-1,n-1} + (n-1+j)a_{j,n-1}, \quad \text{when } 2 \leq j \leq n-1.
    \end{align*}
\end{lem}
\begin{proof}
    The result follows immediately by induction on the exponential $n$ after performing the computation
    \begin{equation*}
        \biggl(x^2\dfrac{\partial}{\partial x}\biggr)^n = x^2\dfrac{\partial}{\partial x}\biggl(x^2\dfrac{\partial}{\partial x}\biggr)^{n-1} = x^2\dfrac{\partial}{\partial x} \sum_{i=1}^{n-1}a_{i,n-1}x^{n-1+i}\dfrac{\partial^i}{\partial x^i}.
   \qedhere
    \end{equation*}
\end{proof}

We now specialize $\Bbbk$ to characteristic $p$.
\begin{prop} \label{prop:pnilp}
	If $R_n$ is a $\kp$-algebra then the derivation $\dif_n$ is $p$-nilpotent.
\end{prop}

\begin{proof}
We already know that $\dif_n^p(f)=0$ when $f\in \Pol_n\subset R_n$, then it is sufficient to show that $\dif_n^p(\omega_i)=0$ for all $i=1,\dotsc n$.
  
 Since $\dif_n=\sum\limits_{r=1}^nx_r^2\dfrac{\partial}{\partial x_r}$ (\autoref{lem:difaspartial}) then 
	\begin{align} \label{cor:d_p_nil1}
		\dif_n^p&=\biggl(\sum\limits_{r=1}^nx_r^2 \dfrac{\partial }{\partial x_r}\biggr)^p=\sum\limits_{\substack{k_1+\cdots +k_n=p \\ k_1,\ldots, k_n \geq 0}}\binom{p}{k_1,k_2,\dotsc , k_n}\prod\limits_{r=1}^{n}\biggl(x_r^2\dfrac{\partial}{\partial x_r}\biggr)^{k_r} 
  =\sum\limits_{r=1}^n\biggl(x_r^2\dfrac{\partial}{\partial x_r}\biggr)^{p}.
	\end{align}
	The last equality follows from the fact that we are working in characteristic $p$ and the multinomial coefficient
        \[
          \binom{p}{k_1,k_2,\dotsc , k_n}=\frac{p!}{k_1!\cdot k_2! \cdots k_n!}
	\]
    is equal to $1$ whenever one of $k_1,\dotsc,k_n$ is $p$, while in all the other cases it is equal to $0 \pmod p$.
    Now, observe that in characteristic $p$ the equality in \autoref{lem:powerx2patial} can be restated as
\begin{align} \label{cor:d_p_nil2}
		\biggl(x^2_r\dfrac{\partial}{\partial x_r}\biggr)^{p}=\sum_{i=2}^{p}a_{i,p}x_r^{p+i}\dfrac{\partial^i}{\partial x_r^i},
\end{align}
    since $a_{1,p}$ is $p! \equiv 0 \pmod p$.
     As a consequence of equations~\eqref{cor:d_p_nil1},~\eqref{cor:d_p_nil2}, and~\autoref{prop:doublepart}, we have that for all $i=1,\dotsc, n$
    \begin{align*}
        \dif_n^p(\omega_i)=\sum_{r=1}^{n}\biggl( \sum_{i=2}^pa_{i,p}x_r^{p+i} \dfrac{\partial^i}{\partial x_r^i} \biggr)\omega_i = 0,
    \end{align*} 
    since all the powers of the partial derivatives are greater than or equal to $2$.
\end{proof}

\subsection{Equivariant homomorphisms}\label{sec:pdgA}

We say that a map between algebras with derivations is \emph{$\dif$-equivariant} if it commutes with the derivations.

\begin{thm} \label{thm:pdgphi}
The map $\phi'_{n}\colon R_{n}\to R_{n+1}$, which is defined as
	\begin{align*}
	\phi_{n}'(\omega_i) &=\omega_i-\sum\limits_{l=i+1}^{n+1}x_{n+1}\prod\limits_{j=i+1}^{l-1}(x_j-x_{n+1})\omega_l, \\
        \phi'_{n}(x_i) &=x_i, \\ 
        \phi'_{n}(yz) &= \phi'_{n}(y)\phi'_{n}(z),
        \end{align*} 
where $i\in \{1,2,\dotsc , n\}$, $y,z\in R_{n}$, and with the convention that $\prod_{j=i+1}^{i}(x_j-x_{n+1})=1$,
is a grading preserving, $\dif$-equivariant homomorphism.
\end{thm}

\begin{proof}
    The map  $\phi'_n$ is grading preserving and it is easily proven that it respects the algebra relations that define $R_n$, thus it is a homomorphism. We have left to prove that it commutes with the derivations of $R_{n}$ and $R_{n+1}$. Since $\phi'_{n}$ commutes with the derivation when restricted to the polynomial sub-algebras of $R_{n}$ and $R_{n+1}$, we just need to prove that $\phi'_{n}(\dif_{n+1}\omega_i)=\dif_n\phi'_{n}(\omega_i)$ on the generators $\omega_i$ for $1\leq i \leq n$. This result is consequence of the equality $\dif_{n+1}=\sum_{j=1}^{n+1}x_j^2\frac{\partial}{\partial x_j}$, and the fact that
	$$\frac{\partial}{\partial x_k}\phi'_{n}(\omega_i)=
	\begin{cases}
		\phi'_{n}(\frac{\partial}{\partial x_k}\omega_i) &\text{if } 1\leq k \leq n ,\\
		0 &\text{if } k=n+1.
	\end{cases}
	$$
	which we will now prove. \\
	 Let $\iota\colon R_{n}\to R_{n+1}$ be the canonical inclusion (it sends $x_i$ to $x_i$ and $\omega_i$ to $\omega_i$ for $i=1,\dotsc ,n$). It commutes with the partial derivatives. Indeed, recall that the action of $\frac{\partial}{\partial x_r}$, as described in~\autoref{defn:partialder}, can be defined on $R_n$ for all $n\geq r$ and it does not depend on $n$. (Observe that $\iota$ does not commute with the derivations $\dif_{n}$, $\dif_{n+1}$).
	The image of $\phi'_{n}$ on the elements of $R_{n}$ of the form $f=\sum_{j=1}^{n}p_j\omega_j$, where $p_j \in \Pol_{n}$, is equal to
	$$
		\phi'_{n}(f)=\Bigl(1-x_{n+1}\frac{\partial}{\partial x_{n+1}}\Bigr)\iota(f).
	$$
	This follows directly from~\autoref{defn:partialder} and the fact that $(1-x_{n+1}\frac{\partial}{\partial x_{n+1}})\iota(-)$ is $\Pol_{n}$-linear. (In general $\phi'_{n}(\omega_i\omega_j)$ is not equal to $(1-x_{n+1}\frac{\partial}{\partial x_{n+1}})\iota(\omega_i\omega_j)$). For $k=1,\dotsc, n$ we compute  
	\begin{align*}
		\frac{\partial}{\partial x_k}\phi'_{n}(\omega_i)&=\frac{\partial}{\partial x_k}\Bigl(1-x_{n+1}\frac{\partial}{\partial x_{n+1}}\Bigr)\iota(\omega_i)=\Bigl(\frac{\partial}{\partial x_k}-x_{n+1}\frac{\partial}{\partial x_k}\frac{\partial}{\partial x_{n+1}}\Bigr)\iota(\omega_i)\\
		&=\Bigl(\frac{\partial}{\partial x_k}-x_{n+1}\frac{\partial}{\partial x_{n+1}}\frac{\partial}{\partial x_k}\Bigr)\iota(\omega_i)=\Bigr(1-x_{n+1}\frac{\partial}{\partial x_{n+1}}\Bigr)\iota\Bigl(\frac{\partial}{\partial x_k}\omega_i\Bigr)\\
		&=\phi'_{n}\Bigl(\frac{\partial}{\partial x_k}\omega_i\Bigr).
	\end{align*}
When $k=n+1$ using~\autoref{prop:doublepart} we get that
\begin{align*}
	\frac{\partial}{\partial x_{n+1}}\phi'_{n}(\omega_i)&=\frac{\partial}{\partial x_{n+1}}\Bigl(1-x_{n+1}\frac{\partial}{\partial x_{n+1}}\Bigr)\iota(\omega_i)=-x_{n+1}\frac{\partial}{\partial x_{n+1}}\frac{\partial}{\partial x_{n+1}}\iota(\omega_1)=0.
\end{align*}
We conclude that, for $1\leq i \leq n$:
	\begin{align*}
		\phi'_{n}(\dif_{n+1}\omega_i)&=\phi'_{n}(\sum_{j=1}^{n+1}x_j^2\frac{\partial}{\partial x_j}\omega_i)=\sum_{j=1}^{n}x_j^2\frac{\partial}{\partial x_j}\phi'_{n}(\omega_i)=
		\dif_{n}\phi'_{n}(\omega_i).
	\end{align*}
Since $\phi'_{n}$ commutes with the derivations $\dif_{n-1}$ and $\dif_n$ when acting on the generators $x_1,\dotsc, x_{n-1}$ and $\omega_1,\dotsc, \omega_{n-1}$ of $R_{n-1}$, we conclude that $\Phi_{n-1,n}$ is a $\dif$-equivariant morphism of algebras.
\end{proof}

\begin{defn}
	For $m$, $n$ positive integers such that $m\leq n$ we define the $\dif$-equivariant homomorphism $\phi'_{m,n}:R_m\rightarrow R_n$ as
	\begin{align*}
	    \phi'_{m,n}(f)&:=\phi'_{n-1}\phi'_{n-2}\cdots \phi'_{m}(f) & \text{if } m<n, \\
        \phi'_{n,n}(f)&:=f  & \text{if } m=n,
	\end{align*}
	for $f\in R_m$.
\end{defn}
Note that $\phi'_{m,n}$ preserves all gradings. 
Observe that, since $\dif_{m}(\omega_m)=0$, then the image of $\omega_m$ under the morphism  $\phi'_{m,n}$ is an element of $\ker(\dif_n)$. This allows us to define a new $\dif$-equivariant morphism between the algebras $\bigl(\bV( \omega_1,\dotsc ,\omega_n), 0 \bigr)$ and $(R_n,\dif_n)$ by sending each $\omega_i$ to the element $\phi'_{i,n}(\omega_i)$ of $R_n$.

\begin{prop} \label{prop:morphOmega}
Let's consider the graded algebra $\bV ( \omega_1,\dotsc ,\omega_n)$ endowed with the trivial derivation. There exists a $\dif$-equivariant, grading preserving homomorphism
    \begin{align}
        \Omega_n\colon \bV ( \omega_1,\dotsc ,\omega_n) &\longrightarrow R_n\\
        \omega_i &\mapsto \Omega_n(\omega_i)\coloneq \phi'_{i,n}(\omega_i). \nonumber
    \end{align}
\end{prop}

\begin{proof}
By the definition of $\Omega_n$ we have that $\dif_n(\Omega_n(\omega_i))=\dif_n(\phi'_{i,n}(\omega_i))=\phi'_{i,n}(\dif_i\omega_i)=0$, which implies that $\Omega_n$ commutes with the derivations.
	Moreover, since for all $1\leq i \leq n$ the morphisms $\phi'_{i,n}$ are even then the elements $\Phi_{i,n}(\omega_i)$ are odd and we have that $\Omega_n(\omega_i\omega_j)=\Omega_n(\omega_i)\Omega_n(\omega_j)=\phi'_{i,n}(\omega_i)\phi'_{j,n}(\omega_i)=-\phi'_{j,n}(\omega_j)\phi'_{i,n}(\omega_i)=-\Omega_n(\omega_i)\Omega_n(\omega_j)=\Omega_n(-\omega_j\omega_i)$, which makes $\Omega_n$ a morphism of algebras.
      \end{proof}
	Recall the the labeled omegas $\omega_k^a$ from~\autoref{defn:labelomega}. If $k=n$ and $a=n-i$ we have that $\omega_n^{n-i}= \sum\limits_{s=0}^{n-i}(-1)^s\hs_{s}(i+s,n)\omega_{i+s}$. The elements $\omega_n^{n-i}$ turn out to be the images of the $\omega_i$ under the map $\Omega_n$.

      \begin{prop} \label{prop:Omega}
  The $\dif$-equivariant homomorphism $\Omega_n:\bigl(\bV ( \omega_1,\dotsc ,\omega_n), 0 \bigr) \longrightarrow (R_n, \dif_n)$, defined in~\autoref{prop:morphOmega}, is such that $$\Omega_n(\omega_i)= \omega_n^{n-i}.$$
\end{prop}

To prove the proposition we need some preparation. Recall the notation  
\begin{lem} \label{lem:hs}
	We have the equality
	$$
	(x_m-x_n)\hs_{l-1}(m,n)+\hs_l(m+1,n)-\hs_l(m,n-1)=0.
	$$
\end{lem}

\begin{proof}
	This equality follows from the easy results
\begin{align*}
\hs_l(m,n)&=x_n\hs_{l-1}(m,n)+\hs_{l}(m,n-1), \\
\hs_l(m,n)&=x_m\hs_{l-1}(m,n)+\hs_{l}(m+1,n).  
\qedhere
\end{align*}
\end{proof}

\begin{lem} \label{lem:hs1}
  Suppose that $n,m,q$ are positive integers such that $q+m\leq n$.
  Then the following equality holds:
	\begin{align*} 
		\sum_{s=0}^{m}(-1)^s\hs_s(q+s,n)\prod_{j=q+1+s}^{q+m}(x_j-x_{n+1})=(-1)^m\hs_{m}(q+m+1,n+1).
	\end{align*}
\end{lem}

\begin{proof}
	We prove the statement by induction on $m$. If $m=1$ the equality is easily verified. Now suppose the equality is true for $m-1$, then
	\begin{align*}
		&\sum_{s=0}^{m}(-1)^s\hs_s(q+s,n)\prod_{j=q+1+s}^{q+m}(x_j-x_{n+1})\\
		&=\sum_{s=0}^{m-1}(-1)^s\hs_s(q+s,n)\prod_{j=q+1+s}^{q+m-1}(x_j-x_{n+1})(x_{q+m}-x_{n+1})+(-1)^m\hs_{m}(q+m,n)\\
		&=(-1)^{m-1}\hs_{m-1}(q+m,n+1)(x_{q+m}-x_{n+1})+(-1)^m\hs_{m}(q+m,n)\\
		&=(-1)^m\hs_m(q+m+1,n+1).
	\end{align*}
We use the induction hypothesis to obtain the second equality, while the last one follows from application of~\autoref{lem:hs}.
\end{proof}

\begin{proof}[Proof of~\autoref{prop:Omega}]
	We proceed by induction on the rank $n$ of the algebras $\bV ( \omega_1,\dotsc ,\omega_n)$ and $R_n$. For $n=1,2$ we can prove the result by straight computation. Suppose now that the claim holds for rank $n$, then:
	\begin{align*}
		\Omega_{n+1}(\omega_i)&=\phi'_{n}\phi'_{n-1}\cdots \phi'_{i}(\omega_i)=\phi'_{n}\Omega_n(\omega_i)=\phi'_{n}(\omega_n ^{n-i})\\
		&=\phi'_{n}\biggl(\sum_{s=0}^{n-i}(-1)^s\hs_{s}(i+s,n)\omega_{i+s}\biggr)\\
		&=\sum_{s=0}^{n-i}(-1)^s\hs_{s}(i+s,n)\phi'_{n,n+1}(\omega_{i+s})\\
		&=\sum_{s=0}^{n-i}(-1)^s\hs_{s}(i+s,n)\biggl(\omega_{i+s}-x_{n+1}\sum_{l=i+s+1}^{n+1}\prod_{j=i+s+1}^{l-1}(x_j-x_{n+1})\omega_l\biggr)\\
		&=\sum_{s=0}^{n-i}(-1)^s\hs_{s}(i+s,n)\omega_{i+s}\\
		& \quad \quad -x_{n+1}\sum_{s=0}^{n-i}\sum_{l=i+s+1}^{n+1}(-1)^s\hs_{s}(i+s,n)\prod_{j=i+s+1}^{l-1}(x_j-x_{n+1})\omega_l.\\
	\end{align*}
Since 
\[
\sum_{s=0}^{n-i}\sum_{l=i+s+1}^{n+1}(-1)^s\hs_{s}(i+s,n) =\sum_{l=i+1}^{n+1}\sum_{s=0}^{l-(i+1)}(-1)^s\hs_{s}(i+s,n) ,
\]
we get
	\begin{align*}
		\Omega_{n+1}(\omega_i)&=\sum_{l=i}^{n}(-1)^{l-i}\hs_{l-i}(l,n)\omega_{l}\\
		& \quad \quad -x_{n+1}\sum_{l=i+1}^{n+1}\sum_{s=0}^{l-(i+1)}(-1)^s\hs_{s}(i+s,n)\prod_{j=i+s+1}^{l-1}(x_j-x_{n+1})\omega_l.\\
		\end{align*}
	We now change the index in the first sum setting $s=l-i$ and rewrite the sum as a linear combination of the omegas:
	\begin{align*}
		\Omega_{n+1}(\omega_i)&=\omega_i+\sum_{l=i+1}^{n} \bigg( (-1)^{l-i}\hs_{l-i}(l,n) \\
		& \quad \quad -x_{n+1}\sum_{s=0}^{l-(i+1)}(-1)^s\hs_{s}(i+s,n)\prod_{j=i+s+1}^{l-1}(x_j-x_{n+1})\bigg)\omega_l\\
		& \quad \quad -x_{n+1}\sum_{s=0}^{n-i}(-1)^s\hs_{s}(i+s,n)\prod_{j=i+s+1}^{n}(x_j-x_{n+1})\omega_{n+1}.\\
	\end{align*}
	Now, using~\autoref{lem:hs1}, we obtain that
		\begin{align*}
		\Omega_{n+1}(\omega_i)&=\omega_i+\sum_{l=i+1}^{n}\bigg((-1)^{l-1}\hs_{l-i}(l,n)-x_{n+1}(-1)^{l-i-1}\hs_{l-i-1}(l,n+1)\biggr)\omega_l\\
		&\quad \quad -x_{n+1}(-1)^{n-i}\hs_{n-i}(n+1,n+1)\omega_{n+1}\\
		&=\omega_i+\sum_{l=i+1}^{n}(-1)^{l-i}\hs_{l-i}(l,n+1)\omega_l-x_{n+1}(-1)^{n-i}\hs_{n-i}(n+1,n+1)\omega_{n+1}\\
		&=\sum_{l=i}^{n+1}(-1)^{l-i}\hs_{l-i}(l,n+1)\omega_{l}=\sum_{s=0}^{n+1-i}(-1)^s\hs_s(i+s,n+1)\omega_{i+s}=\omega_{n+1}^{n+1-i},
	\end{align*}
	which proves the result.
\end{proof}

\begin{cor} \label{cor:symextaction}
	There is a $\dif$-equivariant morphism of algebras
	\begin{align*}
		\bigl( \bV (\omega_n^0,\omega_n^1,\dotsc, \omega_n^{n-1} ), 0\bigr) &\rightarrow (R_n^{\Sy_n},\dif_n)\\ 
		\omega_n^j &\mapsto \omega_n^j \quad \text{ for } \ 0\leq j \leq n-1
	\end{align*}
	which is the inclusion. As a consequence $\dif_n(\omega_n^j)=0$ for $0\leq j \leq n-1.$
\end{cor}
\begin{proof}
From~\autoref{prop:Rnsym} we have the isomorphism of algebras
\[
R_n^{\Sy_n}\cong \Pol_n^{\Sy_n}\otimes \bV (\omega_n^0,\omega_n^1,\dotsc, \omega_n^{n-1} ) .
\]
\autoref{prop:Omega} assures that this isomorphism extends to a $\dif$-equivariant isomorphism when endowing the exterior algebra with the trivial differential (note that this does not violate~\autoref{rel:diszero}). Indeed, the action of $\dif_n$ on the $\omega_n^{j}$ is zero for $j=0,\dotsc,n-1$:
\begin{equation*}
    \dif_n(\omega_n^{n-i})=\dif_n(\Omega_n(\omega_{i}))=0.
    \qedhere
 \end{equation*}
\end{proof}

Clearly all the properties of the $\omega_n^{n-i}$ can be translated into properties for the elements $\Omega_n(\omega_{i})$, in particular this allows a recursive definition of $\Omega_n$:

\begin{prop}
	Fix a strictly positive integer $n$. Here, for $0<j<n$, we consider $\Omega_j(\omega_j)$ as an element of $R_n$ instead of $R_j$ (it is enough to compose the morphism with an inclusion), then for $1\leq i <n$ we have the equality
\[
	\Omega_n(\omega_i)=\Omega_{n-1}(\omega_i)-x_n\Omega_{n}(\omega_{i+1}).
\]

\end{prop}

\begin{proof}
	This is a direct consequence of the recursive definition of the labeled omegas given in~\cite[\S$2.4$]{naissevaz2}.
\end{proof}
Observe that this recursive definition is backward on the index $i$ of the omegas. In particular if we know the element $\Omega_{n}(\omega_{i+1})$ then we know $\Omega_{n-1}(\omega_i)$. indeed $\Omega_{n-1}(\omega_i)$ is equal to $\Omega_{n}(\omega_{i+1})$ with all the indexes lowered by $1$:
\begin{align*}
	&\Omega_{n}(\omega_{i+1})=\sum_{s=0}^{n-i-1}(-1)^s\hs_s(i+1+s,n)\omega_{i+1+s},\\
	&\Omega_{n-1}(\omega_{i})=\sum_{s=0}^{n-i-1}(-1)^s\hs_s(i+s,n-1)\omega_{i+s}.
\end{align*}
	We have already shown in~\autoref{prop:difelem} how to express the polynomials coefficients $\alpha_{i,j}$, that appears in the definition of the derivation $\dif_n(\omega_i)=\sum_{j=i+1}^n\alpha_{i,j}\omega_j$, using the elementary symmetric polynomials. Using~\autoref{prop:Omega} we can rewrite them in terms of the complete symmetric polynomials:
\begin{cor}
  The polynomials $\alpha_{i,j}$ introduced in~\autoref{def:alpha_ij} satisfy the recursive formula
\begin{align*}
	\alpha_{n-1,n} &=\dif_n(\hs_1(n,n))=x_n^2,\\
	\alpha_{i,j} &=\sum_{s=1}^{j-i-1}(-1)^{s-1}\hs_s(i+s,n)\alpha_{i+s,j}+(-1)^{j-i-1}\dif_n(\hs_{j-i}(j,n)) .
\end{align*}
\end{cor}

\begin{proof}
	from~\autoref{prop:Omega} we have
	\begin{align*}
		\dif_n(\Omega_n(\omega_i))=\dif_n\biggl(\sum_{s=0}^{n-i}(-1)^s\hs_s(i+s,n)\omega_{i+s}\biggr)=0 ,
	\end{align*}
and therefore 
	\begin{align*}
		\dif_n(\omega_i)
        &=\sum_{s=1}^{n-i}(-1)^{s-1}\bigl(\dif_n(\hs_s(i+s,n))\omega_{i+s}+\hs_s(i+s,n)\dif_n(\omega_{i+s})
		\bigr)\\
		&=\sum_{s=1}^{n-i-1}(-1)^{s-1}\biggl(\dif_n(\hs_s(i+s,n))\omega_{i+s}+\hs_s(i+s,n)\sum_{j=i+s+1}^{n}\alpha_{i+s,j}\omega_j\biggr)\\
		& \quad \quad +(-1)^{n-i-1}\dif_n(\hs_{n-i}(n,n))\omega_{n}\\
		&=\sum_{j=i+1}^{n-1}(-1)^{j-(i+1)}\dif_n(\hs_{j-i}(j,n))\omega_{j}+\sum_{j=i+2}^{n}\sum_{s=1}^{j-(i+1)}(-1)^{s-1}\hs_s(i+s,n)\alpha_{i+s,j}\omega_j\\
		&\quad \quad +(-1)^{n-i-1}\dif_n(\hs_{n-i}(n,n))\omega_{n}\\
		&=\dif_n(\hs_1(i+1,n))\omega_{i+1}\\
		&\quad \quad +\sum_{j=i+2}^{n-1}\biggl((-1)^{j-(i+1)}\dif(\hs_{j-i}(j,n))+\sum_{s=1}^{j-(i+1)}(-1)^{s-1}\hs_s(i+s,n)\alpha_{i+s,j}\biggr)\omega_j\\
		&\quad \quad +(\sum_{s=1}^{n-(i+1)}(-1)^{s-1}\hs_s(i+s,n)\alpha_{i+s,n}+(-1)^{n-i-1}\dif_n(\hs_{n-i}(n,n)))\omega_n.
	\end{align*}
	Since $\alpha_{i,j}$ is the polynomial coefficient of $\omega_j$ we see that the statement is verified.
\end{proof}
Until now we have endowed $R_n$ with a derivation $\dif_n$ which commutes with the action of the symmetric group, making it a well defined derivation for the subalgebra $R_n^{\Sy_n}$. At the very beginning of this section we introduced the balanced free left $(R_n^{\Sy_n}, \dif_n)$-module $\cR_n=R_n.v$, which is generated by $v$ and with a derivation such that $\dif_n(v)=\sum_{i=1}^n\alpha_ix_iv$, where the coefficients $\alpha_i$ are such that the difference $\alpha_{i+1}-\alpha_i=a$ is a fixed constant. The endomorphism algebra of $\cR_n$ is the algebra $A_n=\End_{R^{\Sy_n}}(\cR_n)$. The derivation $\dif_n$ defined on $\cR_n$ allows us to define a derivation $\dif_a$ on $A_n$, depending on the constant $\alpha_{i+1}-\alpha_i=a$, by the formula:
\[
\dif_a(\xi)(f.v)=\dif_{n}(\xi(f.v))-\xi(\dif_n(f.v)) .
\]
From this it follows that $\dif_a$ depends on $n$ and is equal to the derivation $\dif_n$ on the subalgebra $R_n$, while its action on $T_i$ is described in~\eqref{eq:pdgdemazure}. The above can be summarized in the following (note that we have taken $a=1$ in~\eqref{eq:pdgdemazure}, which is what we usually do in practice, but the results in the rest of this section hold for any $a$).
\begin{prop}\label{prop:Andiff}
The algebra $A_n$ admits a derivation $\dif_a$ which is defined on generators by 
    \begin{align*}
	   \dif_a(x_i) &= x_i^2,\\
	   \dif_a(\omega_i) &= \sum_{j=i+1}^n\alpha_{i,j}\omega_j,\\
	   \dif_a(T_i) &= a-(a+1)x_iT_i+(a-1)x_{i+1}T_i.
    \end{align*}
When the ground field is $\kp$ this differential is $p$-nilpotent, making $(A_n,\dif_a)$ a $p$-dg algebra.
\end{prop}

\begin{prop} \label{prop:inclusion}
	There exists a grading preserving homomorphism of algebras $\phi_{n}$ from $A_{n}$ to $A_{n+1}$ which is injective and intertwines the derivation. It is defined on generators by
	\begin{align*}
		\phi_{n}(\omega_i)&=\omega_i-\sum\limits_{l=i+1}^{n+1}x_{n+1}\prod\limits_{j=i+1}^{l-1}(x_j-x_{n+1})\omega_l &&\text{ for } 1\leq i \leq n,\\[2ex]
		\phi_{n}(x_i)&=x_i  &&\text{ for } 1\leq i \leq n,\\[2ex]
		\phi_{n}(T_j)&=T_j    &&\text{ for } 1\leq j \leq n-1,
	\end{align*}
	where we use the convention that $\prod_{j=i+1}^{i}(x_j-x_{n+1})=1$.
\end{prop}

\begin{proof}
	The map $\phi_{n}$ is clearly a $\dif$-equivariant morphism of algebras when restricted on the nilHecke subalgebras $\nh_n$ because it coincides with the inclusion $\nh_{n}\hookrightarrow \nh_{n+1}$  and the derivation acts on the nilHecke generators locally (see~\eqref{eq:pdgx} and~\eqref{eq:pdgdemazure}). Moreover $(R_{n},\dif_n)$ is a subalgebra of $(A_{n},\dif_a)$, and $\phi_{n}$ coincides with $\phi'_{n}$ when restricted to $R_{n}$. It follows that $\phi_{n}(\dif_{n}\omega_i)=\dif_{n+1}\phi_{n}(\omega_i)$, as proven in~\autoref{prop:Omega}. Then, $\phi_{n}$ commutes with the derivation on the generators of $A_{n}$, and therefore on every element of $A_{n}$.\\
    The same reasoning shows that $\phi_{n}$ respects relations~\eqref{eq:nh1},~\eqref{eq:nh2},~\eqref{eq:nh3},~\eqref{eq:nhs1} and~\eqref{eq:nhs2}.
	Then, to prove that $\phi_{n}$ is a morphism of algebras we just need to show that 
 \[
\phi_{n}(T_i(\omega_i-x_{i+1}\omega_{i+1}))=\phi_{n}((\omega_i-x_{i+1}\omega_{i+1})T_i) . 
 \]
Observe that $i<n+1$ and
	\begin{align*}
		&\phi_{n}(T_i(\omega_i-x_{i+1}\omega_{i+1}))=T_i(\phi_{n}(\omega_i)-x_{i+1}\phi_{n}(\omega_{i+1}))=\\
		&=T_i(\omega_i-x_{i+1}\omega_{i+1})-T_i(\sum_{l=i+1}^{n+1}x_{n+1}\!\!\prod_{j=i+1}^{l-1}(x_j-x_{n+1})\omega_l-x_{i+1}\sum_{l=i+2}^{n+1}x_{n+1}\!\!\prod_{j=i+2}^{l-1}(x_j-x_{n+1})\omega_l)\\
		&=(\omega_i-x_{i+1}\omega_{i+1})T_i-T_i\Bigl(\sum_{l=i+1}^{n+1}x_{n+1}\!\!\prod_{j=i+1}^{l-1}(x_j-x_{n+1})\omega_l-x_{i+1}\sum_{l=i+2}^{n+1}x_{n+1}\!\!\prod_{j=i+2}^{l-1}(x_j-x_{n+1})\omega_l\Bigr) .
	\end{align*}
    Since
  \begin{align*}
		&\sum_{l=i+1}^{n+1}x_{n+1}\prod_{j=i+1}^{l-1}(x_j-x_{n+1})\omega_l-x_{i+1}\sum_{l=i+2}^{n+1}x_{n+1}\prod_{j=i+2}^{l-1}(x_j-x_{n+1})\omega_l\\
		&= x_{n+1}\omega_{i+1}+(x_{i+1}-x_{n+1})\sum_{l=i+2}^{n}x_{n+1}\!\!\prod_{j=i+2}^{l-1}(x_j-x_{n+1})\omega_l-x_{i+1}\sum_{l=i+2}^{n+1}x_{n+1}\!\!\prod_{j=i+2}^{l-1}(x_j-x_{n+1})\omega_l\\
		&=x_{n+1}\Bigl(\omega_{i+1}-\sum_{l=i+2}^{{n+1}}x_{n+1}\prod_{j=i+2}^{l-1}(x_j-x_{n+1})\omega_l\Bigr),
	\end{align*}
we have 
    \begin{align*}
      &T_i\Bigl(\sum_{l=i+1}^{{n+1}}x_{n+1}\prod_{j=i+1}^{l-1}(x_j-x_{n+1})\omega_l-x_{i+1}\sum_{l=i+2}^{n+1}x_{n+1}\prod_{j=i+2}^{l-1}(x_j-x_{n+1})\omega_l\Bigr)\\
        &=\Bigl(\sum_{l=i+1}^{{n+1}}x_{n+1}\prod_{j=i+1}^{l-1}(x_j-x_{n+1})\omega_l-x_{i+1}\sum_{l=i+2}^{n+1}x_{n+1}\prod_{j=i+2}^{l-1}(x_j-x_{n+1})\omega_l\Bigr)T_i ,
    \end{align*}
and therefore
    \[
        \phi_{n}(T_i(\omega_i-x_{i+1}\omega_{i+1}))=\phi_{n}((\omega_i-x_{i+1}\omega_{i+1})T_i) . 
    \]

    We have left to prove that $\phi_{n}$ is injective. Let $v\neq 0$ be an homogeneous element of $A_{n}$ of degree $(q,2\lambda)$ where $\lambda \geq0$. Then
    $$
        v=\sum_{1\leq i_1< i_2 < \cdots < i_{\lambda} \leq n-1} v_{(i_1,\ldots, i_\lambda)}\omega_{i_1}\cdots \omega_{i_\lambda},
    $$
    where $v_{(i_1,\ldots, i_\lambda)}$ are homogeneous elements of $\nh_{n}$. If $\lambda=0$ then $v=v_0\in \nh_{n}$. We order the ordered $\lambda$-tuples $(i_1,\ldots,i_\lambda)$ using the lexicographic order and we denote the smallest one by $(i'_1,\ldots,i'_\lambda)$. Then
    $$
        \phi_{n}(v)=\sum_{1\leq i_1< i_2 < \cdots < i_{\lambda} \leq n-1} v_{(i_1,\ldots, i_\lambda)}\omega_{i_1}\cdots \omega_{i_\lambda}+\substack{\text{Elements indexed by} \\ \text{strictly greater ordered} \\ \text{$\lambda$-tuples than $(i'_1,\ldots,i'_\lambda)$}} \ ,
    $$
    which implies that $\phi_{n}(v)\neq 0$. Due to the fact that $\phi_{n}$ is grading preserving we can conclude that $\ker(\phi_{n})=\{0\}$.
\end{proof}

The pair $(\dif_n,\phi_n)$ are not necessarily unique (see~\autoref{ex:notuniq}). We don't investigate other derivations
\begin{ex}\label{ex:notuniq}
Recall that for each $n$, the derivation is defined by its action on $\omega_1$. 
For example, for $n=2,3$ one could take
\[
\dif'_2(\omega_1) = -x_1^2\omega_2 ,\mspace{60mu}\dif'_2(\omega_2) = 0 ,
\]
and
\begin{gather*}
\dif'_3\omega_3 = 0 ,
\mspace{60mu}
\dif'_3\omega_2 = x_3^2\omega_3 + \alpha, 
\\
\dif'_3\omega_1 = (x_2^2+x_3^2)\omega_2+x_3^2(x_2-x_3)\omega_3+\beta -x_1\alpha ,
\end{gather*}
with $\alpha \in R_3^{\Sy_3}$, of $q$-degree $-2$, and $\beta \in R_3^{\Sy_3}$ of $q$-degree $0$,
and proceed recursively. We leave to the reader to find the appropriate homomorphism $\phi'_n$. 
\end{ex}


%% file: files/cataction.tex
\label{sec:cataction}

\subsection{Baby Verma modules}

Let $q\in\bC$ be a primitive root of unity of order $p$.

\begin{defn} \label{def:smallqg}
    The \emph{small quantum group} $\smslt$ is the unital $\bC$-algebra with generators
 $E$, $F$, $K^{\pm 1}$ and relations
\begin{gather*}
    KK^{-1}=K^{-1}K=1,\qquad  KE=q^2EK, \qquad KE=q^2EK,
    \\
    EF-FE=\frac{K-K^{-1}}{q-q^{-1}},\qquad E^p=0,\qquad F^p=0. 
\end{gather*}
This is a finite-dimensional algebra and has a Hopf algebra structure (see for example~\cite[\S9.3A-\S9.3B]{chari-pressley}).
\end{defn}
We are interested in a particular representation of $\smslt$ which is called baby Verma module in~\cite[\S3.2]{AndersenMazorchuk}. 
Denote by $b_q\subset\smslt$ the standard Borel subalgebra and let $\lambda$ be a formal parameter. 
Le also $\bC_\lambda$ be a 1-dimensional representation of the Cartan subalgebra extended trivially to $b_q$. 
\begin{defn}
The \emph{baby Verma module} $M(\lambda)$ is the induced module
 \[
M(\lambda) := \smslt\otimes_{b_q}\bC_\lambda .
\]   
\end{defn}
The module $M(\lambda)$ is a $p$-dimensional $\bC[\lambda^{\pm 1}]$-module and can be obtained as the specialization to $\smslt$ of the module $V(\lambda,0,0)$ given in~\cite[Ex. 11.1.7]{chari-pressley}.
Its rational form $M(\lambda)\otimes_{\bC[\lambda]}\bC(\lambda)$ is irreducible by the results of~\cite[\S11]{chari-pressley}.
It has a $\bC[\lambda^{\pm 1}]$-basis $\{m_0,\dotsc, m_{p-1}\}$ where the action of $\smslt$ is given by 
\begin{align*}
	Km_i &=\lambda q^{-2i}m_i ,\\[1ex]
	Fm_i &=
	\begin{cases}
		m_{i+1} &\text{ if }\ 0\leq i < p-1 ,\\
		0, &\text{ if }\ i=p-1 ,
	\end{cases}\\[1ex]
	Em_{i} &=
	\begin{cases}
		0 & \text{ if }\ i=0 ,\\
		[i]\frac{\lambda q^{1-i}-\lambda^{-1}q^{i-1}}{q-q^{-1}}m_{i-1} & \text{ if }\ 0 < i \leq p-1 , 
\end{cases}
\end{align*}
(cf.~\cite[Ex. 11.1.7]{chari-pressley}). Here $[i]=q^{i-1}+q^{i-3}+ \cdots + q^{3-i}+q^{1-i}$ is the quantum number.

\begin{rem}
A baby Verma module occurs as the finite-dimensional submodule generated by the highest weight vector of the universal Verma module for $U_q(\slt)$ after specializing $q$ to a root of unity.    
\end{rem}

\subsection{The Grothendieck group of \texorpdfstring{$(A_n,\dif)$}{(A,d)}}
\label{sec:GgroupAn}

Throughout this section we work over $\Bbbk_p$.
We now show that under particular hypothesis the Grothendieck group $\mathbf{K}_0(A_n)\cong \bO_p$ (\autoref{cor:Kgroup}). We will prove it by establishing that the derived categories $\cD(A_n)$ and $\cD(\bV(\und{\omega}_n^a)\otimes \Bbbk)$ are equivalent, and that $\mathbf{K}_0(\bV(\und{\omega}_n^a)\otimes \Bbbk)\cong \bO_p$. 

\vspace{0.1in}

Let $\cR_n^+ \coloneq R_n\cdot v_+$ be the $p$-dg $R_n$-module generated by $v_+$ and endowed with the differential $\dif_+ (v_+)=\sum_{i=1}^n(i-n)x_iv_+$. Since $R_n^{\Sy_n}\subset R_n$ is a $p$-dg subalgebra, $\cR_n^+$ is also a $R_n^{\Sy_n}$ $p$-dg module which is free with the set $\{ x_1^{b_1}x_2^{b_2}\cdots x_n^{b_n}v_+ \ | 0\leq b_i \leq n-i \}$ as a $\dif_+$-stable basis. 
The vector space $U_n^+\coloneq \Bbbk\left\langle x_1^{b_1}x_2^{b_2}\cdots x_n^{b_n}v_+ \ | 0\leq b_i \leq n-i \right\rangle$ is $\dif_+$-stable, and therefore the isomorphisms in~\eqref{eq:R_niso} are $p$-dg isomorphisms of $p$-dg modules
\[
	\cR_n^+\cong R_n^{\Sy_n} \otimes U_n^+ \cong \Pol_n^{\Sy_n}\otimes \bV(\und{\omega}_n^a)\otimes U_n^+.
\]

We define the $p$-dg algebra $(A_n,\dif)$ to be the $p$-dg endomorphism algebra
\[
A_n\coloneq \End_{R_n^{\Sy_n}}(\cR_n^+).
\]
As a consequence of the stability of $U_n^+$ under the action of the differential $\dif_+$, the isomorphism~\eqref{eq:Andecomposition} becomes an isomorphism of $p$-dg $R_n^{\Sy_n}$-modules 
\begin{align*}
	A_n\coloneq \End_{R_n^{\Sy_n}}(\cR_n^+) \cong R_n^{\Sy_n}\otimes \End_\Bbbk(U_n^+) \cong \Pol_n^{\Sy_n}\otimes \bV (\und{\omega}_n^a)\otimes \mat(n!,\Bbbk).
\end{align*}
Remember from~\autoref{cor:symextaction}, the differential $\dif_+$ acts trivially on  $\bV(\und{\omega}_n^a)$.
The differential in $U_n^+$ induces a differential on $\End_\Bbbk(U_n^+)\cong \mat(n!,\Bbbk)$ which characterize the $p$-dg structure of $A_n$. This is clearly the same differential that makes the following isomorphisms
\begin{align*}
	\nh_n \coloneq \End_{\Pol_n^{\Sy_n}}(\Pol_n.v_+)\cong \Pol_n^{\Sy_n}\otimes \End_\Bbbk(U_n^+) \cong \Pol_n^{\Sy_n}\otimes \mat(n!,\Bbbk),
\end{align*}
being $p$-dg. As a consequence, isomorphism~\eqref{eq:AnNhdecomposition} becomes an isomorphism of $p$-dg algebras
\begin{equation} \label{eq:iso3}
	A_n \cong \bV (\und{\omega}_n^a) \otimes \nh_n.
\end{equation}

Moreover, as showed in~\cite[\S$2.3$]{KhQi} every derivation on the matrix algebra arises as taking a commutator with some fixed matrix. In particular for a $p$-nilpotent derivation the matrix is a block matrix where each block is a Jordan block.
If all the jordan blocks are of size less than or equal to $p$ and at least one block is of dimension strictly less then $p$, then the column module $\mat(n!,\Bbbk)E_{n!,n!}$ is a compact cofibrant generator of $\cD(\mat(n!,\Bbbk))$~\cite[Proposition $2.43$]{KhQi} and it induces an equivalence of categories (\cite[Corollary $2.45$]{KhQi})
\[
  \mat(n!,\Bbbk)E_{n!,n!}\otimes(-):\cD(\Bbbk) \rightarrow \cD(\mat(n!,\Bbbk)).
\]
Since $\nh_n$ and $\mat(n!,\Bbbk)$ are quasi-isomorphic and the category $\cD(\nh_n)$ is equivalent to $\cD(\Pol_n^{\Sy_n})$, we can state the following important result, which makes possible most of the results that will follow.

\begin{prop}
  For any $1\leq n <p$ the differential induced by the $p$-dg algebras $\nh_n$ and $A_n$ on the matrix algebra $\mat(n!,\Bbbk)$ is defined by a $n!\times n!$ block matrix where
\begin{enumerate}
\item  each block is a Jordan block with dimension less than or equal to $p$ and,
\item at least one block has dimension strictly less than $p$.
\end{enumerate}
  If $n\geq p$ the algebra $A_n$ is acyclic.
\end{prop}
\begin{proof}
	The first statement of the proposition follows directly from the considerations made just above. While, when $n\geq p$ the $p$-dg algebra $\nh_n$ is acyclic (\cite[Proposition $3.15$]{KhQi}) and as a consequence of isomorphism~\eqref{eq:iso3} $A_n$ is also acyclic.
\end{proof}
Consider the ideal $\widebar{\Pol}_n^{\Sy_n}$ of $\Pol_n^{\Sy_n}$ which consists of the symmetric polynomials without constant term. The set
$\widebar{\Pol}_n^{\Sy_n}\cdot \cR_n^+$ is a $R_n^{\Sy_n}$-submodule of $\cR_n^+$, we can thus consider its quotient $\cR_n^+ /\left(\widebar{\Pol}_n^{\Sy_n}\cdot \cR_n^+ \right)\cong U_n^+\otimes \bV(\und{\omega}_n^a)$. Since $\widebar{\Pol}_n^{\Sy_n}\cdot \cR_n^+$ is stable under the action of $\dif_+$ the quotient module $\cR_n^+ /\left(\widebar{\Pol}_n^{\Sy_n}\cdot \cR_n^+\right)$ is $p$-dg, moreover it is easy to see that the isomorphism $\cR_n^+ /\left(\widebar{\Pol}_n^{\Sy_n}\cdot \cR_n^+ \right)\cong U_n^+\otimes \bV(\und{\omega}_n^a)$ lifts to a $p$-dg isomorphism, where $U_n^+$ and $\bV(\und{\omega}_n^a)$ are endowed with the usual differentials. Since each endomorphism of $\cR_n^+$ preserves the submodule $\widebar{\Pol}_n^{\Sy_n}\cdot \cR_n^+$ we get a $\dif_+$-equivariant projection
\[
    \pi_n : A_n=\End_{R_n^{\Sy_n}}(\cR_n^+) \rightarrow \End_{R_n^{\Sy_n}}(\cR_n^+/\widebar{\Pol}_n^{\Sy_n}\cdot \cR_n^+)\cong \End_\Bbbk(U_n^+)\otimes \bV (\und{\omega}_n^a).
\]
    
\begin{prop}
	The projection
	\[
		\pi_n:A_n \proj \bV(\und{\omega}_n^a)\otimes \mat(n!,\Bbbk)
	\]
	is a quasi-isomorphism of $p$-dg algebras.
\end{prop}
\begin{proof}
	It follows from the isomorphism of $p$-dg algebras
	\[
		A_n\cong \Pol_n^{\Sy_n}\otimes \bV (\und{\omega}_n^a)\otimes \mat(n!,\Bbbk),
	\]
	and the fact that $\Pol_n^{\Sy_n}$ is quasi-isomorphic to $\Bbbk$.
\end{proof}
\begin{cor}
	There is an equivalence of triangulated categories
	\begin{equation} \label{eq:iso1}
		\pi_n^\star :\cD(A_n) \rightarrow \cD(\bV(\und{\omega}_n^a)\otimes \mat(n!,\Bbbk)).
	\end{equation}
\end{cor}

\begin{lem} \label{lem:exactfunctor}
	The functors
\begin{align*}
	\bV(\und{\omega}_n^a)\otimes(-) &:\mat(n!,\Bbbk)\amod \rightarrow \bV(\und{\omega}_n^a)\otimes \mat(n!,\Bbbk)\amod,
 \\
	\bV(\und{\omega}_n^a)\otimes(-) &:\mat(n!,\Bbbk)_\dif\amod \rightarrow (\bV(\und{\omega}_n^a)\otimes \mat(n!,\Bbbk))_\dif\amod, 
\end{align*}
 are exact.
\end{lem}

\begin{proof}
	Consider the composition of $\bV(\und{\omega}_n^a)\otimes(-)$ with the forgetful functor
	\[
	\mat(n!,\Bbbk)\amod \xrightarrow{\bV(\und{\omega}_n^a)\otimes (-)} \bV(\und{\omega}_n^a)\otimes \mat(n!,\Bbbk)\amod \rightarrow \Bbbk\amod.
	\]
	Since $\bV(\und{\omega}_n^a)$ is a finite dimensional $\Bbbk$-vector space this composition is an exact functor. This implies that $\bV(\und{\omega}_n^a)\otimes (-)$ is exact because being an injective or surjective function is a set-theoretic concept.
	To prove that $\bV(\und{\omega}_n^a)\otimes (-)$ is an exact $p$-dg functor it is enough to show that the image of a $p$-dg morphism is a $p$-dg morphism. This is trivial, indeed for $f$ a $p$-dg morphism in $\mat(n!,\Bbbk)_\dif\amod$ we have
	\begin{align*}
		\left(\bV(\und{\omega}_n^a)\otimes f\right)(\dif (a\otimes b))&=\left(\id\otimes f\right) (a\otimes \dif b)=\id(a)\otimes f(\dif b)\\
		& =\id(a)\otimes \dif f( b)=\dif (\id(a)\otimes f( b)). \qedhere
 \end{align*}
\end{proof}

\begin{prop} \label{prop:compcof}
	If $1\leq n < p$, the $p$-dg $\bV(\und{\omega}_n^a)\otimes \mat(n!,\Bbbk)$-module $\bV(\und{\omega}_n^a)\otimes \mat(n!,\Bbbk)E_{n!,n!}$ is compact and cofibrant. 
\end{prop}

\begin{proof}
the $p$-dg $\mat(n!,\Bbbk)$-module $\mat(n!,\Bbbk)E_{n!,n!}$ is compact and cofibrant (\cite[Proposition $2.43$]{KhQi}). As a consequence of~\cite[Proposition $2.24$ v)]{KhQi} there exist a $p$-dg $\mat(n!,\Bbbk)$-module $N$ such that $\mat(n!,\Bbbk)E_{n!,n!}\oplus N$ is a finite cell module, that is, there is a finite exhaustive filtration of $p$-dg $\mat(n!,\Bbbk)$-modules
	\[
	0=F_{-1}\subset F_0 \subset \cdots \subset F_r \subset F_{r+1} \subset \cdots \subset F_s=\mat(n!,\Bbbk)E_{n!,n!}\oplus N
	\]
	such that the quotients $F_{r+1}/F_r$ are isomorphic, as $\mat(n!,\Bbbk)$-modules, to finite direct sums of the free $\mat(n!,\Bbbk)$-modules $\mat(n!,\Bbbk)$ with shifted degrees, for all $0<r<s$. 
 
 We now consider the $p$-dg $\bV(\und{\omega}_n^a)\otimes \mat(n!,\Bbbk) $-module $\bV(\und{\omega}_n^a)\otimes N$. Clearly
	\[
	\left(\bV(\und{\omega}_n^a)\otimes \mat(n!,\Bbbk)E_{n!,n!}\right)\oplus (\bV(\und{\omega}_n^a)\otimes N) \cong \bV(\und{\omega}_n^a)\otimes(\mat(n!,\Bbbk)E_{n!,n!}\oplus N).
	\]
	The module $\bV(\und{\omega}_n^a)\otimes(\mat(n!,\Bbbk)E_{n!,n!}\oplus N)$ turns out to be a finite cell $\bV(\und{\omega}_n^a)\otimes \mat(n!,\Bbbk)$-module. Indeed, by defining $\bar{F}_r \coloneq \bV(\und{\omega}_n^a)\otimes F_r$, for all $0\leq r \leq s$, we obtain a finite exhaustive filtration of $p$-dg $\bV(\und{\omega}_n^a)\otimes \mat(n!,\Bbbk)$-modules:
	\[
	0=\bar{F}_{-1}\subset F_0 \subset \cdots \subset \bar{F}_r \subset \bar{F}_{r+1} \subset \cdots \subset \bar{F}_s=\bV(\und{\omega}_n^a)\otimes \mat(n!,\Bbbk)E_{n!,n!}\oplus N.
	\]
	Since $\bV(\und{\omega}_n^a)\otimes(-)$ is an exact functor, as shown in~\autoref{lem:exactfunctor}, from the short exact sequence of $\mat(n!,\Bbbk)$-modules
	\[
	0\rightarrow F_r \rightarrow F_{r+1} \rightarrow  \frac{F_{r+1}}{F_r}\rightarrow 0,
	\]
	we obtain the short exact sequence of $\bV(\und{\omega}_n^a)\otimes \mat(n!,\Bbbk)$-modules
	\[
	0\rightarrow \bar{F}_r \rightarrow \bar{F}_{r+1} \rightarrow \bV(\und{\omega}_n^a)\otimes \frac{F_{r+1}}{F_r}\rightarrow 0,
	\]
	and $\bar{F}_{r+1}/\bar{F}_r\cong \bV(\und{\omega}_n^a)\otimes (F_{r+1}/F_r)$. Because $F_{r+1}/F_r$ is a finite direct sum of the free $\mat(n!,\Bbbk)$-modules $\mat(n!,\Bbbk)$ with shifted degrees it follows that $\bar{F}_{r+1}/\bar{F}_r$ is isomorphic to a finite direct sum of the free $\bV(\und{\omega}_n^a)\otimes \mat(n!,\Bbbk)$-modules $\bV(\und{\omega}_n^a)\otimes \mat(n!,\Bbbk)$ with shifted degrees.
	We have just shown that the $p$-dg $\bV(\und{\omega}_n^a)\otimes \mat(n!,\Bbbk)$-module $\bV(\und{\omega}_n^a)\otimes \mat(n!,\Bbbk)E_{n!,n!}$ is a direct summand of a finite cell module, which implies that it is compact cofibrant.
\end{proof}

\begin{prop} \label{prop:compcogen}
	If $1\leq n < p$, the $p$-dg $\bV(\und{\omega}_n^a)\otimes \mat(n!,\Bbbk)$-module $\bV(\und{\omega}_n^a)\otimes \mat(n!,\Bbbk)E_{n!,n!}$ is a compact cofibrant generator of $\cD(\bV(\und{\omega}_n^a)\otimes \mat(n!,\Bbbk))$.
\end{prop}

\begin{proof}
	Still in~\cite[Proposition $2.43$]{KhQi} it has been proven that $\mat(n!,\Bbbk)$ has a filtration of $p$-dg modules
	\[
		0\subset F_1 \subset \cdots \subset F_{n!}=\mat(n!,\Bbbk) 
	\]
	such that $F_{r+1}/F_r$ is isomorphic to the column module $\mat(n!,\Bbbk)E_{n!,n!}$. Defining $\bar{F}_r=\bV(\und{\omega}_n^a)\otimes F_r$ gives a filtration of $p$-dg modules
	\[
		0\subset \bar{F}_1 \subset \cdots \subset \bar{F}_{n!}=\bV(\und{\omega}_n^a)\otimes \mat(n!,\Bbbk).
	\]
	Tensoring the short exact sequence obtained from the first filtration
	\[
		0 \rightarrow F_r \rightarrow F_{r+1} \rightarrow \frac{F_r}{F_{r+1}} \rightarrow 0
	\]
	we obtain the exact sequence
	\[
		0 \rightarrow \bar{F}_r \rightarrow \bar{F}_{r+1} \rightarrow \bV(\und{\omega}_n^a)\otimes\frac{F_r}{F_{r+1}} \rightarrow 0.
	\]
	As a consequence the quotients $\bar{F}_{r+1}/\bar{F}_{r}$ are isomorphic to $\bV(\und{\omega}_n^a)\otimes \mat(n!,\Bbbk)E_{n!,n!}$ with the grading shifted. This implies that $\bV(\und{\omega}_n^a)\otimes \mat(n!,\Bbbk)E_{n!,n!}$ is a compact generator of $\cD(\bV(\und{\omega}_n^a)\otimes \mat(n!,\Bbbk))$.
\end{proof}

\begin{prop}
	The functor 
	\begin{equation} \label{eq:iso2}
		(\bV(\und{\omega}_n^a)\otimes \mat(n!,\Bbbk)E_{n!,n!})\otimes(-):\cD(\bV(\und{\omega}_n^a)\otimes \Bbbk) \rightarrow \cD(\bV(\und{\omega}_n^a)\otimes \mat(n!,\Bbbk))
	\end{equation}
	is an equivalence of derived categories.
\end{prop}

\begin{proof}
	This follows from~\cite{KhQi}[Proposition $2.34$]. Indeed the $p$-dg $\bV(\und{\omega}_n^a)\otimes \mat(n!,\Bbbk)$-module $\bV(\und{\omega}_n^a)\otimes \mat(n!,\Bbbk)E_{n!,n!}$ is a compact cofibrant generator of $\cD(\bV(\und{\omega}_n^a)\otimes \mat(n!,\Bbbk))$, as shown in~\autoref{prop:compcof} and~\autoref{prop:compcogen}.\\
	Moreover, observe that 	as $p$-dg algebras
	\begin{align*}
	\End_{\Lambda(\und{\omega}_n^a)\otimes \mat(n!,\Bbbk)}(\bV(\und{\omega}_n^a)\otimes \mat(n!,\Bbbk)E_{n!,n!})& \cong \bV(\und{\omega}_n^a)\otimes \End_{\mat(n!,\Bbbk)}(\mat(n!,\Bbbk)E_{n!,n!})\\
	&\cong \bV(\und{\omega}_n^a)\otimes \Bbbk . \qedhere
	\end{align*}
\end{proof}

\begin{prop}
	The Grothendieck group
	\begin{equation}
		\mathbf{K}_0(\bV(\und{\omega}_n^a)\otimes \Bbbk)\cong \mathbb{O}_p .
	\end{equation}
\end{prop}

\begin{proof}
Note that $\bV(\und{\omega}_n^a)\otimes \Bbbk$ is nonzero only in non-positive $q$-degrees, with $\Bbbk$ in $q$-degree zero. 
 In particular, this means that it has only one idempotent, and thus its derived category has a unique compact cofibrant generator (up to isomorphism and grading shift).
This generator is given by the unique finite cell module, which, in this case, coincides with the unique projective indecomposable module. 
This means that $\bV(\und{\omega}_n^a)\otimes \Bbbk$ satisfies the conditions where the the arguments given in the discussion previous to \cite[Corollary $2.18$]{EliasQi1} apply (the fact that $\bV(\und{\omega}_n^a)\otimes \Bbbk$ is negatively graded is innocuous). 
It follows that 
	\[
		\mathbf{K}_0(\bV(\und{\omega}_n^a))\cong \mathbb{O}_p 
	\]
and, since $\bV(\und{\omega}_n^a)\otimes \Bbbk$ is quasi-isomorphic to $\bV(\und{\omega}_n^a)$, this completes the proof.
\end{proof}

\begin{cor} \label{cor:Kgroup}
	For $0\leq n < p$ the Grothendieck group of $A_n$ is one dimensional 
	\[
	\mathbf{K}_0(A_n)\cong \mathbb{O}_p.
	\]
\end{cor}

\begin{proof}
	From the equivalences of triangulated categories~\eqref{eq:iso1} and~\eqref{eq:iso2} it follows that there exists an equivalence
	\begin{equation*}
		\cD(A_n) \rightarrow \cD(\bV(\und{\omega}_n^a)\otimes \Bbbk).
	\end{equation*}
	This implies that
	\begin{equation*}
		\mathbf{K}_0(A_n)\cong\mathbf{K}_0(\bV(\und{\omega}_n^a)\otimes \Bbbk) \cong \mathbb{O}_p. \qedhere
	\end{equation*}
\end{proof}

\subsection{A categorical action}

The inclusion of $p$-dg algebras $\phi_{n-1}$ defined in~\autoref{prop:inclusion} turns $A_n$ into a $p$-dg module over $A_{n-1}$.

Recall the morphism $\phi_{n-1}':R_{n-1}\rightarrow R_n$ from~\autoref{thm:pdgphi}.
\begin{prop} \label{prop:R_nmod}
Morphism $\phi_{n-1}'$ makes $R_n$ a $R_{n-1}$-module with basis $\{x_n^{a}\omega_n^{b}\ |\ a \in \bN,\ b\in \{0,1\} \}$. Since $\phi_n'$ is a $p$-dg morphism, $R_n$ is a $p$-dg $R_{n-1}$-module.
\end{prop}

\begin{proof}
We start by showing that the set $\{x_n^{a}\omega_n^{b}\ |\ a \in \bN,\ b\in \{0,1\} \}$ generates the $R_{n-1}$-module $R_n$. The action defined by $\phi_n'$ is just the usual multiplication for the elements of the polynomial ring contained in $R_{n-1}$, thus the polynomial part is clearly generated by the elements of the set proposed as basis. we now show that the elements $\omega_1,\dotsc , \omega_n$ are generated as well. For $i=0,\dotsc, n-1$ we define the subalgebras $R_{n-1}^{>i}\coloneq\Pol_{n-1} \otimes \bV(\omega_{i+1},\dotsc, \omega_{n-1})$ of $R_{n-1}$. We have inclusions
	\[
		R_{n-1}^{>n-1}=\Pol_{n-1} \subset R_{n-1}^{>n-2} \subset \cdots \subset R_{n-1}^{>0}=R_{n-1}.
	\]
	It is easy to show the last equality of the following formula using induction
	\begin{align*}
		\omega_i\cdot 1=\omega_i-x_n\sum_{j=i+1}^n\prod_{s=i+1}^{j-1}(x_s-x_n)\omega_j=\omega_i+\substack{\text{elements generated} \\ \text{by the action of $R_{n-1}^{>i}$}.}
	\end{align*}
	As a consequence, using the action of $R_{n-1}^{>i}$ we can obtain $\omega_i$:
	\[
		\omega_i=\omega_i\cdot 1+\substack{\text{elements generated} \\ \text{by the action of $R_{n-1}^{>i}$}.}
	\]
	Since all the generators of $R_n$ can be obtained using the action of $R_{n-1}$ on the set $\{x_n^{a}\omega_n^{b}\ |\ a \in \bN,\ b\in \{0,1\} \}$, it is easy to observe that all the elements can be generated as well.\\
	We now prove that the set $\{x_n^{a}\omega_n^{b}\ |\ a \in \bN,\ b\in \{0,1\} \}$ is linearly independent, that is, it is a basis. Let $r$ be a strictly positive integer, $f_1,\dotsc f_r \in R_{n-1}$ and $a_i\in \bN$, $b_i \in \{0,1\}$ for $1\leq i \leq r$, and suppose that
	\begin{align*}
		f_1\cdot x_n^{a_1}\omega_n^{b_1}+f_2\cdot x_n^{a_2}\omega_n^{b_2}+\cdots + f_r\cdot x_n^{a_r}\omega_n^{b_r}=0 .
	\end{align*}
	We will prove that $f_1=\cdots=f_r=0$.
	Consider the set of pairs $\{(a_s,b_s)\ | \ s=1,\dotsc,r\}$, which can be ordered using the lexicographic order. Without loss of generality we can suppose that $(a_1,b_1)<(a_2,b_2)<\cdots <(a_r,b_r)$. Now, consider a monomial $mx_n^{a}\omega_n^b$, with $m\in R_{n-1}$, if we define its degree to be $(a,b)$ then we observe that
	\[
		f_i\cdot x_n^{a_i}\omega_n^{b_i}=f_i x_n^{a_i}\omega_n^{b_i}+ \text{Elements of higher degree.}
	\]
	It is clear that if $f_1\neq 0$ then $f_1 x_n^{a_1}\omega_n^{b_1}$ can not be eliminated by the other terms of the sum $f_1\cdot x_n^{a_1}\omega_n^{b_1}+\cdots + f_r\cdot x_n^{a_r}\omega_n^{b_r}$. Extending this reasoning to all the $f_i$ we conclude that $f_1=\cdots =f_r=0$.
\end{proof} 

Recall the morphism $\phi_{n-1}:A_{n-1}\rightarrow A_n$ from~\autoref{prop:inclusion}.
\begin{prop}
Morphism $\phi_{n-1}$  makes $A_n$ an $A_{n-1}$-module with basis $\{x_n^{a}\omega_n^{b}T_{s_{n-1}s_{n-2}\cdots s_c}\ |\ a \in \bN,\ b\in \{0,1\}, c\in\{1, 2,\dotsc, n-1\} \}$, with $T_{s_{n-1}s_n}=1$ by convention. Since $\phi_{n-1}$ is a $p$-dg morphism, $A_n$ is a $p$-dg $A_{n-1}$-module.
\end{prop}

\begin{proof}
The statement is a consequence of~\autoref{prop:R_nmod} together with the fact that there exists an isomorphism of algebras $A_n\cong R_n\otimes \nc_n$, by using the right coset decomposition of $\Sy_n$, i.e.
	\[
		\Sy_n=\bigsqcup_{c=1}^{n}\Sy_{n-1}s_{n-1}\cdots s_c,
	\]
    with the convention that $\Sy_{n-1}s_{n-1}s_n=\Sy_{n-1}$.
\end{proof}

For a $\bZ/2\bZ\times\bZ^2$-graded module $M$ we use the notation $q^a\lambda^b\Pi M $, where its $q$-degree is shifted up by $a$, its $\lambda$-degree is shifted up by $b$, and its parity is switched.

\vspace{0.1in}

Consider the $(A_n,A_{n-1})$-bimodule $A_n$, where the right action of $A_{n-1}$ is induced by the morphism $\phi_{n-1}$.
\begin{defn}
 We define the functor
	\begin{align*}
		\cF_n^{n+1}\coloneq A_{n+1}\otimes_{\phi_{n}}(-):\cD(A_n)&\longrightarrow \cD (A_{n+1}),\\
		M&\mapsto  A_n\otimes_{\phi_{n}}M.
	\end{align*}
\end{defn}
	Consider the $(A_{n-1},A_n)$-bimodule $A_n^t$ where the left action of $A_{n-1}$ is defined by the morphism $\phi_{n-1}$ and the differential acts on the generator $1$ as $\dif(1)=(n-1)x_n1$.
\begin{defn}
We define the functor
	\begin{align} \label{def:E}
		\cE_{n}^{n-1}\coloneq q^{2n-1}\lambda ^{-1}A_n^t\otimes (-):\cD(A_n)&\longrightarrow \cD (A_{n-1}),\\
		M&\mapsto -q^{2n-1}\lambda ^{-1} A_n^t\otimes_{\phi_{n-1}}M.\nonumber
	\end{align}
\end{defn}

Recall that $\cD(A_n)=0$ when $n\geq p$.
\begin{defn}
We define
\begin{gather*}
  \cD(A) \coloneq\bigoplus_{n=0}^{+\infty}\cD(A_n)=\bigoplus_{n=0}^{p-1}\cD(A_n)  ,
\\
	\cF\coloneq \bigoplus_{n=0}^{+\infty}\cF_n , \mspace{60mu}  \cE\coloneq \bigoplus_{n=0}^{+\infty}\cE_n ,
\end{gather*}
and
\[
(A) \coloneq \bigoplus_{n=0}^{p-1}\mathbf{K}_0(A_n)\cong \bigoplus_{n=0}^{p-1}\mathbb{O}_p[A_n],
\]
where the symbols $[A_n]$ form a basis of the $\mathbb{O}_p$ vector space $\mathbf{K}_0(A)$.
\end{defn}

\begin{prop} \label{weakcat}
	In the Grothendieck group $\mathbf{K}_0(A)$ we have
	\begin{align*}
		[\cF(A_n)] &=\begin{cases}
			[A_{n+1}], &\text{ if }\  0\leq n < p-1,\\
			0, &\text{ if }\   n = p-1.
		\end{cases}
  \\[1ex]
		[\cE(A_n)] &=
		\begin{cases}
			0, &\text{ if }\  n = 0,\\
			[n](\lambda q^{-n}-\lambda ^{-1} q^{n}) &\text{ if }\  0< n \leq p-1.
		\end{cases}
	\end{align*}
\end{prop}

\begin{proof}
	The $p$-dg module $A_n^t\otimes_{\phi_{n-1}}A_n$ is the direct sum of two $p$-dg $A_{n-1}$-modules we denote $F$ and $F^\omega$ for the duration of the proof. 
 Set 
 \[
\underline{T}_i = 
\begin{cases}
T_{n-1}T_{n-2}\dotsm T_{i} & \text{for } i=1,2,\dotsc,n-2 ,
\\
T_{n-1} & \text{for } i=n-1 ,
\\
1 & \text{for } i=n .
\end{cases}
  \]
The modules $F$ and $F^\omega$ are generated respectively by the sets
	\begin{gather*}
		\bigl\{ 1\underline{T}_ix_{n}^r \ |\ r \in \mathbb{N},\ i\in 	\{1,2,\dotsc,n\} \bigr\},
  \\
		\bigl\{ 1\underline{T}_ix_{n}^r\omega_{n} \ |\ r \in \mathbb{N},\ i\in \{1,2,\dotsc,n\} \bigr\},	    
	\end{gather*}

The $p$-dg module $F$ has a filtration
\[
		F_0\coloneq 0\subset F_1 \subset \cdots F_{n-1} \subset F_n, 
\]
where $F_j$ is generated by the set 
\[
\bigl\{ 1\underline{T}_ix_{n}^r \ |\ r \in \mathbb{N},\ i\in 	\{n-(j-1),\dotsc,n\} \bigr\} .
\]
The quotient $F_j/F_{j-1}$ is a $p$-dg $A_{n-1}$-module generated by the classes $\left[1\underline{T}_{j-1}x_{n}^r\right]$, where
\[
\bigl[ \dif (1\underline{T}_{j-1}x_{n}^r )\bigr] = \bigl((n-1)-2(j-1)+r \bigr)\bigl[ 1\underline{T}_{j-1}x_{n}^{r+1}\bigr].
\]
When $r=2(j-1)-(n-1)\bmod p=2j-1-n \bmod p$ we have that 
\[
\bigl[ \dif (1 \underline{T}_{j-1}x_{n}^r )\bigr]=0  ,
\]
and this implies that  $F_j/F_{j-1}$ is the direct sum of the submodules $M_j^k$ generated by the elements of the sets
\[
\bigl\{ \bigl[ 1 \underline{T}_{j-1}x_{n}^r \bigr] | r=0,\dotsc , 2j-1-n\bmod p \bigr\} 
\mspace{40mu} \text{for $k=0$},
\]
and
\[
\bigl\{ \bigl[ 1\underline{T}_{j-1}x_{n}^r \bigr] | r= 2j-1-n+kp+1,\dotsc , 2j-1-n+(k+1)p \bigr\} 
\mspace{40mu}\text{for $k>0$} .
\]
All modules $M_j^{k\geq1}$ are contractible and therefore acyclic. The only module that is not contractible is $M_j^0$ when $2j-1-n \neq p-1 \bmod p$. 
Since $2j-1-n=p+2j-1-n\bmod p$ we can write
\[
\left[M_j^0\right]+\left[M_{n-j}^0\right] =\sum_{s=0}^{p+2j-1-n}q^{2(1-j)+2s}+ \sum_{s=0}^{n-2j-1}q^{2(j+1-n)+2s}=\sum_{s=0}^{p-1}q^s=0.
\]
The last equality follows from the fact that $q$ is a $p$-root of unity and a simple computation of the coefficients. Observe that is $n$ is even and $j=\frac{n}{2}$ then $M_{j}=M_{n-j}=M_{\frac{n}{2}}$, but in this case $2j-1-n \bmod p=p-1\bmod p$ and the module $M_{\frac{n}{2}}$ turns out to be acyclic.

We conclude that 
\[
[F]=\sum_{j=1}^{n}\left[\frac{F_j}{F_{j-1}}\right]=\sum_{j=1}^{n}\left[M_j^0\right]=\left[M_n^0\right]=\sum_{j=0}^{n-1}q^{2(1-n)+2j} ,
\]
where the last equality follows from the definition of $\dif$ on $M_n^0$ and its action on the generators. Similarly we prove that
\[
[F^{\omega}]=-\lambda^{2}q^{-2n}\sum_{j=0}^{n-1}q^{2(1-n)+2j} ,
\]
where the $-\lambda^{-2}q^{2n}$ is a consequence of the degree shift of the generators of $F^{\omega}$ given by the element $\omega_n$. In particular the minus sign is a consequence of $\bZ/2\bZ$-graded structure of the module. 
 To complete the proof we recall the grading shift in the definition of $\cE_{n}$ in~\eqref{def:E} that results in a factor $-q^{2n-1}\lambda^{-1}$.
 \end{proof}

\begin{cor}
Functors $\cF$ and $\cE$ descend to operators on the Grothendieck group of $A$, which is isomorphic to the baby Verma module $M(\lambda q^{-1})$ via the map that sends $[A_r]$ to $(q-q^{-1})m_r$.
\end{cor}


%% file: files/sl2.tex
\label{sec:sltwo}

\subsection{Witt actions}

The partial derivatives from~\autoref{sec:partder} work integrally and lift to Witt operators on $\rnz$. 
For each $k\in\{-1,0,1,2,\dots\}$ define 
\[ 
\ell_k = \sum\limits_{j=1}^{n}x_j^{k+1}\dfrac{\partial }{\partial x_j} . 
\]
The action of the $\ell_k$s on $\rnz$ satisfy $[\ell_k,\ell_r]=(r-k)\ell_{k+r}$.

\vspace{0.1in}

In the sequel we consider $k\in\{0,\pm 1\}$ and we define operators $\ee=\ell_1=\dif_n$, $\ff=-\ell_{-1}$, and $\hh=[\ee,\ff]=2\ell_0$.

\subsection{The \texorpdfstring{$\cslt$}{sl2}-module \texorpdfstring{$\rnz$}{Rn}}\label{ssec:Rn-slt}
As in~\cite{eliasqi-sl2}  we use the notation $(\Pol_n,\cslt)$, $(\rnz,\cslt)$, etc... to emphasize that we are looking at the $\cslt$-module structure.

\begin{rem}
Contrary to the case of $\Pol_n$  the action of the operator $\hh$ on $\rnz$ is not diagonalizable. 
For example for $n=2$ one has 
\begin{align*}
\hh(\omega_1) &= 2x_2\omega_2, & \hh(\omega_2) &= 0 ,
\intertext{and}
\ee(\omega_1) &= x_2^2\omega_2, &    \ee(\omega_2) &= 0
\\
\ff(\omega_1) &= -\omega_2, &  \ff(\omega_2) &= 0 .
\end{align*}
 The action of $\hh$ cannot be be diagonalized but as we will see below, $\rnz$ has a filtration whose successive quotients are weight modules. 
\end{rem}

\begin{rem}
Recall that the \emph{degree operator} $\deg_q$ acts as multiplying any element by its degree.
As the preceding example shows, $\hh$ does not act on $R_n$ as the degree operator, opposite to its action on $\Pol_n$ (and on $\nh_n$).
Nevertheless, the degree operator satisfies
\begin{gather}
\label{eq:degedegf}
    [\deg_q,\ee] = 2\ee ,\mspace{40mu}    [\deg_q,\ff] = -2\ff ,
    \\   
\label{eq:deghh}
    [\deg_q,\hh] = 0 .    
\end{gather}
Relations~\eqref{eq:degedegf} hold in any graded category, whenever derivations of degrees $\pm 2$ are present, while relation~\eqref{eq:deghh} shows that the Lie algebra $\cglt$ (with basis elements $\ee$, $\ff$, $\hh$ and $\deg_q$) acts on $R_n$ (and therefore on $A_n$). 
\end{rem}

Recall the \emph{co-Verma module} $\nabla(0)$ from~\cite[\S 1.6]{eliasqi-sl2}:
\[
\begin{tikzpicture}[scale=.5,baseline=-.08cm,tinynodes]
\node at (0,0) {\small $\dotsc$}; 
\node at (-3,0) {\small $v_{k+1}$}; 
\node at (-6,0) {\small $v_k$}; 
\node at (-9,0) {\small $\dotsc$}; 
\node at (-12,0) {\small $v_2$}; 
\node at (-15,0) {\small $v_1$}; 
\node at (-18,0) {\small $v_0$}; 
\draw [myred,thick,latex-] (-15.2,0.55) to [bend right] (-17.8,0.55);\node at (-16.5,1.5) {$0$};
\draw [myred,thick,latex-] (-12.2,0.55) to [bend right] (-14.8,0.55);\node at (-13.5,1.5) {$\dif=1$};
\draw [myred,thick,latex-] (-3.2,0.55) to [bend right] (-5.8,0.55);\node at (-4.5,1.5) {$\dif=k$};
\draw [myred,thick,latex-] (-0.2,0.55) to [bend right] (-2.8,0.55);\node at (-1.5,1.5) {$\dif=k+1$};
\draw [mygreen,thick,latex-] (-17.8,-0.5) to [bend right] (-15.2,-0.5);
\node at (-16.5,-1.5) {$z=1$};
\draw [mygreen,thick,latex-] (-14.8,-0.5) to [bend right] (-12.2,-0.5);
\node at (-13.5,-1.5) {$z=2$};
\draw [mygreen,thick,latex-] (-5.8,-0.5) to [bend right] (-3.2,-0.5);
\node at (-4.5,-1.5) {$z=k+1$};
\draw [mygreen,thick,latex-] (-2.8,-0.5) to [bend right] (-0.2,-0.5);
\node at (-1.5,-1.5) {$z=k+2$}; 
\end{tikzpicture}
\]
We have that $(\bZ[x],\cslt)\cong\nabla(0)$ via $v_m\mapsto x^m$,  and that
$(\bZ[x_1,\dotsc ,x_n],\cslt)\cong\nabla(0)^{\otimes n}$. 

\vspace{0.1in}

For $\und{\omega}$ a monomial in $\bV^{\mspace{-4mu}m}(\omega_{1},\dotsc, \omega_{n})$ we write $\und{\omega}\nabla(0)$ for the $\cslt$-module
$\und{\omega}\bZ \otimes \nabla(0)$, with $\cslt$ acting trivially on $\und{\omega}\bZ$. 

\vspace{0.1in}

Before considering the general case it is instructive to look closer at the $\cslt$-modules $\rnz$ for $n=1,2$. 
\begin{ex}
We have $R_1=\bZ[x,\omega]/(\omega^2)$ with $\ee = x^2\dfrac{\partial}{\partial x}$, 
$\hh = 2x\dfrac{\partial}{\partial x}$ and $\ff = -\dfrac{\partial}{\partial x}$.
Since $\dfrac{\partial}{\partial x}\omega=0$ the element $\omega x$ generates the co-Verma module $\omega\nabla(0)$ (it is isomorphic to $\nabla(0)$ as ungraded modules via $v_m\mapsto \omega x^m$). 
We have that 
\[ (R_1,\cslt)\cong \nabla(0)\oplus\omega\nabla(0) . \]
\end{ex}

This decomposition along the $\lambda$-grading is no coincidence as the operators $\ee$, $\ff$ and $\hh$ respect the $\lambda$-grading. Moreover, since the submodules of $\rnz$ in the bottom and top $\lambda$-degrees are of rank 1 over $\bZ[x_1,\dotsc , x_n]$ they give always direct summands.

\begin{ex}
We have $R_2=\bZ[x_1,x_2,\omega_1,\omega_2]/(\omega_i\omega_j+\omega_j\omega_i, i,j=1,2)$.
We expect to have at least three summands here. 
In $\lambda$-degree zero we have $R_2^{0}=\bZ[x_1,x_2]$ and therefore
  \[
(R_2^{0}, \cslt) \cong \nabla(0)^{\otimes 2} .
  \]
The case $\lambda=4$ is also trivial: we have  $R_2^{4}=\bZ[x_1,x_2]\omega_1\omega_2$ and
  \[
    \dfrac{\partial}{\partial x_1}(\omega_1\omega_2) = 0 = \dfrac{\partial}{\partial x_2}(\omega_1\omega_2)
   ,
  \]
and so
\[
  (R_2^{4}, \cslt) \cong \omega_1\omega_2\nabla(0)^{\otimes 2}
 .
\]
At this time
\[ (R_2,\cslt)\cong \nabla(0)^{\otimes 2}\oplus (R_2^{2},\cslt) \oplus \omega_1\omega_2\nabla(0)^{\otimes 2}, \]
as expected.
Note that the summands $\nabla(0)^{\otimes 2}$ and $\omega_1\omega_2\nabla(0)^{\otimes 2}$ are weight.

Let us now analyse  $(R_2^{2},\cslt)$.
We have that
  \[
    \dfrac{\partial}{\partial x_1}(\omega_1)=0,\quad
    \dfrac{\partial}{\partial x_2}(\omega_1)=\omega_2, \quad
    \dfrac{\partial}{\partial x_1}(\omega_2) = 0 = \dfrac{\partial}{\partial x_2}(\omega_2)
    .
  \]
Since as abelian groups $R_2^{2}=\bZ[x_1,x_2]\omega_1\oplus \bZ[x_1,x_2]\omega_2$ we see that 
$(\bZ[x_1,x_2]\omega_2,\cslt)$ is a submodule isomorphic to $\omega_2\nabla(0)^{\otimes 2}$ (this is no coincidence, as we will see later). 

One more remark: it is clear that the quotient $(R_2^{2},\cslt)/\omega_2\nabla(0)^{\otimes 2}$ is isomorphic to $\omega_1\nabla(0)^{\otimes 2}$, since $\ee$, $\ff$ and $\hh$ send $\omega_1$ to a multiple (possible zero) of $\omega_2$!
We see that the non-weight module $(R_2^{2},\cslt)$ has a length 1 filtration with quotient isomorphic to the co-Verma module $\nabla(0)^{\otimes 2}$.
\end{ex}

\vspace{0.1in}

Any element of  $\bV^{\mspace{-4mu}m}(\omega_{1},\dotsc, \omega_{n})$ 
can be written as a linear combination  of monomials of the form $\omega_{i_1}\dotsm\omega_{i_m}$ with $i_1<\dotsm < i_m$. The lexicographic order on the set of these monomials introduces a (total) order $<$ on $\bV^{\mspace{-4mu}m}(\omega_{1},\dotsc, \omega_{n})$. 
The minimal and the maximal elements  of $\bV^{\mspace{-4mu}m}(\omega_{1},\dotsc, \omega_{n})$ with respect to $<$ are respectively
$\und{\omega}_1:=\omega_1\dotsm \omega_{m}$
and $\und{\omega}_{\binom{n}{m}}:=\omega_{n-m+1}\dotsm\omega_{n}$.
List all these monomials in order as
\[
\und{\omega}_1 <  \und{\omega}_2 < \dotsc < \und{\omega}_{\binom{n}{m}}
 .
\]
The order $<$ induces a partial order on $(\rnz)^{2m}$ where for nonzero $p,q\in\Pol_n$ we have that 
$p\und{\omega}_i< q\und{\omega}_j$ if $\und{\omega}_i< \und{\omega}_j$.

\begin{ex}
For $(R_3^{2},\cslt)$ we have
\[
  \und{\omega}_1=\omega_1\omega_2 ,\qquad
  \und{\omega}_2=\omega_1\omega_3 ,\qquad
  \und{\omega}_3=\omega_2\omega_3
,
\]
and
\begin{align*}
  \dfrac{\partial}{\partial x_1}(\und{\omega}_1) &= 0,\quad
&  \dfrac{\partial}{\partial x_2}(\und{\omega}_1) &= 0, \quad
&  \dfrac{\partial}{\partial x_3}(\und{\omega}_1) &= \und{\omega}_2 - (x_2-x_3) \und{\omega}_3 ,
\\
  \dfrac{\partial}{\partial x_1}(\und{\omega}_2) &= 0,\quad
&  \dfrac{\partial}{\partial x_2}(\und{\omega}_2) &= \und{\omega}_3,\quad
&  \dfrac{\partial}{\partial x_3}(\und{\omega}_2) &= \und{\omega}_3  , 
\\
  \dfrac{\partial}{\partial x_1}(\und{\omega}_3) &= 0,\quad
&  \dfrac{\partial}{\partial x_2}(\und{\omega}_3) &= 0 ,\quad
&  \dfrac{\partial}{\partial x_3}(\und{\omega}_3) &= 0 . 
\end{align*}
Consider the following submodules of $(R_3^{2},\cslt)$:
\begin{align*}
  (R_3^{2}(\geq\! 3),\cslt) &= (\und{\omega}_3\!\Pol_3^\bZ,\cslt),
                             \\
  (R_3^{2}(\geq\! 2),\cslt) &= (\und{\omega}_2\Pol_3^\bZ\oplus\ \und{\omega}_3\Pol_3^\bZ,\cslt),
                             \\
  (R_3^{2}(\geq\! 1),\cslt) &= (\und{\omega}_1\Pol_3^\bZ\oplus\ \und{\omega}_2\Pol_3^\bZ\oplus\ \und{\omega}_3\Pol_3^\bZ,\cslt)
.
\end{align*}
We have a filtration by $\cslt$-modules
\[
  (R_3^{2}(\geq\!3),\cslt)\subset (R_3^{2}(\geq\! 2),\cslt) \subset (R_3^{2}(\geq\! 1),\cslt)=(R_3^{2},\cslt)
.
\]
Note that the successive quotients in this filtration are 
\[
  (R_3^{2}(\geq\! 2),\cslt)/(R_3^{2}(\geq\!3),\cslt)\cong\und{\omega}_2\nabla(0)^{\otimes 3}
  ,
\]
and
\[
  (R_3^{2}(\geq\! 1),\cslt)/(R_3^{2}(\geq\!2),\cslt)\cong\und{\omega}_1\nabla(0)^{\otimes 3}
  .
\]
\end{ex}

\vspace{0.1ex}  

We now state and prove the main result about the structure of $(\rnz,\cslt)$.
\begin{thm}\label{thm:Rn-slt}
  Let $(\rnz)^{2m}(\geq\mspace{-5mu}r)\subset (\rnz)^{2m}$ be the $\bZ$-submodule spanned by $\Pol_n$ and by all monomials $\und{\omega}_{s}\geq \und{\omega}_{r}$.
\begin{enumerate}
\item  We have a decomposition of $\cslt$-modules 
  \[
(\rnz, \cslt) \cong \bigoplus_{m=0}^{n}((\rnz)^{2m}, \cslt) . 
  \]
\item The summand $\bigl((\rnz)^{2m}, \cslt\bigr)$  has a filtration by $\cslt$-modules
\[
M_{\tbinom{n}{m}-1} \subset M_{\tbinom{n}{m}-2}\subset\dotsm\subset M_1\subset M_0 = (\rnz)^{2m} ,
\]
where $ M_i = \bigl((\rnz)^{2m}(\geq\!j+1), \cslt\bigr)$.
\item Each sucessive quotient is isomorphic to a co-Verma module: 
\[
  M_{r+1}/ M_r \cong \und{\omega}_r\nabla(0)^{\otimes n}
  .
\]
\end{enumerate}
\end{thm} 
\begin{proof}
The first statement is immediate since the $\cslt$-action on $\rnz$ preserves the $\lambda$-grading. 
From~\autoref{rem:partder-rec} it follows that none of the operators $\ee$, $\ff$ and $\hh$ can increase $<$ and this implies that $((\rnz)^{2m}(\geq\mspace{-5mu}r),\cslt)$ is submodule of $((\rnz)^{2m}(\geq\mspace{-5mu}s),\cslt)$ if $r>s$.  To prove the third  statement we observe that as a $\bZ$-module space,
  $(\rnz)^{2m}(\geq\mspace{-5mu}s) / (\rnz)^{2m}(\geq\mspace{-5mu}s+1)$ is isomorphic to $\und{\omega}_s\Pol_n$. 
\end{proof}

\subsection{A filtration on the \texorpdfstring{$\cslt$}{sl2}-module \texorpdfstring{$A_n$}{An}}

To extend the action of $\cslt$ to $A_n$ 
we declare that operator $\ee$ acts on the $T_i$s as $\dif$ in~\eqref{eq:dTi}, while
\begin{equation}
  \ff(T_i)=0, \qquad \hh(T_i)=2T_i,
\end{equation}
for $i=1,\dotsc ,n-1$.
These operators are required to satisfy the Leibniz rule. 

\vspace{0.1in}

Let $(A_n^{2m},\cslt)\subset (A_n,\cslt)$ be the $\cslt$-submodule spanned by all elements in $\lambda$-degree $2m$. We have
\[
(A_n, \cslt) \cong \bigoplus_{m=0}^{n}(A_n^{2m}, \cslt) . 
\]

Denote by $\bigl(A_n^{2m}(\geq\!\! j),\cslt\bigr)\subset (A_n^{2m},\cslt)$ the $\cslt$-submodule generated by $\nh_n$ and $\und{\omega}_r$ for $r\geq j$ and let $(\und{\omega}_s\!\nh_n,\cslt)$ be the $\cslt$-module $\und{\omega}_s\bZ\otimes\nh_n$  with trivial $\cslt$-action on $\und{\omega}_s\bZ$. 

 \begin{prop}
We have a filtration of $(A_n^{2m},\cslt)$ by $\cslt$-modules 
\[
\bigl(A_n^{2m}(\geq\! \tbinom{n}{m}-1),\cslt\bigr)\subset\dotsm \subset\bigl(A_n^{2m}(\geq\! 1),\cslt\bigr) = (A_n^{2m},\cslt) ,
\]
with
\[
\bigl(A_n^{2m}(\geq\! s),\cslt\bigr) / \bigl(A_n^{2m}(\geq\! s+1),\cslt\bigr) \cong (\und{\omega}_s\!\nh_n,\cslt) .
\]
\end{prop}

\subsection{The \texorpdfstring{$\csln$}{sln}-module \texorpdfstring{$\rnz$}{Rn}}

The operators 
\[
\ee_i= x_i\dfrac{\partial}{\partial x_{i+1}}, \qquad\ff_i=x_{i+1}\dfrac{\partial}{\partial x_i}, \qquad\hh_i= [\ee_i,\ff_i]=x_i\dfrac{\partial}{\partial x_{i+1}}-x_{i+1}\dfrac{\partial}{\partial x_i}
\]
define an $\csln$-action on the polynomial ring $\Pol_n$ and it is well known that each (polynomial) degree is irreducible: the linear subspace consisting of polynomials of degree $k$ is isomorphic to the irreducible of highest weight $(k,0,\dotsc ,0)$. 
This action, which is not $\Sy_n$-equivariant, extends to an action on $\rnz$ that preserves all gradings, but each degree component is not irreducible anymore.

Denote by $\mathcal{m}_n$ the set of all monomials in $\bV \left(\omega_{1},\dotsc, \omega_{n}\right)$ and let $R_{n,m}\subset R_n$ be the linear subespace spanned by all monomials of $q$-degree $m$.
Computations in low rank suggest the following. 
\begin{conj}
The $\csln$-module $R_{n,m}$ decomposes as 
\[
R_{n,m} \cong \bigoplus\limits_{\underline{\omega}\in\mathcal{m}_n} V_{\underline{\omega}}, 
\]
where $V_{\underline{\omega}}$ is isomorphic to the irreducible with highest weight 
$\bigl(\tfrac{1}{2}m+\tfrac{1}{2}\qdeg(\underline{\omega}),0,\dotsc,0\bigr)$.
\end{conj}
